\documentclass{amsart}

\usepackage{amsmath, amssymb, amsfonts, amsthm, amscd, amsopn, amstext, amsxtra, euscript}

\usepackage{graphicx} % For images and rotatebox
\usepackage{xcolor}   % For color support

\usepackage[utf8]{inputenc}

\usepackage{enumitem}
\usepackage{booktabs}

\usepackage[all]{xy}

\usepackage[initials, msc-links]{amsrefs}  % alphabetical by author, given names as initials;
\usepackage{hyperref}
\hypersetup{
  colorlinks   = true,
  urlcolor     = blue,
  linkcolor    = blue,
  citecolor    = blue,
  pdftitle     = {Stacky Heights, Entropy, and Zeta Functions of Weighted Hypersurfaces over Finite Fields},
  pdfauthor    = {Sajad Salami and Tanush Shaska},
  pdfsubject   = {Arithmetic geometry; weighted projective stacks over finite fields},
  pdfkeywords  = {weighted projective stack, inertia stack, twist, zeta function, height, entropy, finite field}
}

\usepackage[capitalise]{cleveref}

\newtheorem{thm}{Theorem}[section]
\newtheorem{lem}[thm]{Lemma}
\newtheorem{prop}[thm]{Proposition}
\newtheorem{cor}[thm]{Corollary}

\theoremstyle{definition}
\newtheorem{defn}[thm]{Definition}
\newtheorem{exa}[thm]{Example}

\theoremstyle{remark}
\newtheorem{rem}[thm]{Remark}

\crefname{thm}{Thm.}{Thms.}
\Crefname{thm}{Thm.}{Thms.}
\crefname{prop}{Prop.}{Props.}
\Crefname{prop}{Prop.}{Props.}
\crefname{lem}{Lem.}{Lems.}
\Crefname{lem}{Lem.}{Lems.}
\crefname{cor}{Cor.}{Cors.}
\Crefname{cor}{Cor.}{Cors.}
\crefname{rem}{Rem.}{Rems.}
\Crefname{rem}{Rem.}{Rems.}
\crefname{defn}{Def.}{Defs.}
\Crefname{defn}{Def.}{Defs.}
\crefname{exa}{Ex.}{Exs.}
\Crefname{exa}{Ex.}{Exs.}
\crefname{question}{Question}{Questions}
\Crefname{question}{Question}{Questions}
\crefname{problem}{Problem}{Problems}
\Crefname{problem}{Problem}{Problems}
\crefname{conj}{Conj.}{Conjs.}
\Crefname{conj}{Conj.}{Conjs.}

\newcommand{\cH}{\mathcal{H}}
\newcommand{\cP}{\mathcal{P}}

\newcommand{\G}{\mathbb{G}}

\newcommand{\R}{\mathbb{R}}

\newcommand{\Q}{\mathbb{Q}}
\newcommand{\F}{\mathbb{F}}

\newcommand{\bA}{\mathbb{A}}

\newcommand{\Z}{\mathbb{Z}}

\newcommand{\bP}{\mathbb{P}}

\newcommand{\cL}{\mathcal{L}}

\newcommand{\cC}{\mathcal{C}}
\newcommand{\Oo}{{\mathcal O}}

\newcommand{\cV}{\mathcal{V}}
\newcommand{\cW}{{\mathcal W}}

\newcommand{\e}{\zeta}

\newcommand{\w}{\mathbf{w}}
\newcommand{\x}{\mathbf{x}}

\newcommand{\y}{\mathbf{y}}

\newcommand{\Aut}{\operatorname{Aut}}

\newcommand{\ch}{\operatorname{char}}
\newcommand{\iso}{\cong}  % <-- Restored here!

\newcommand{\stack}{\operatorname{stack}}

\newcommand{\Orb}{\operatorname{Orb}}

\DeclareMathOperator\wgcd{wgcd}

\DeclareMathOperator{\lcm}{lcm}
\DeclareMathOperator\Div{Div}
\DeclareMathOperator\Pic{Pic}
\DeclareMathOperator\Supp{Supp}

\DeclareMathOperator\ord{ord}

\DeclareMathOperator\spec{Spec}

\DeclareMathOperator\Ent{Ent}

\DeclareMathOperator\can{can}
\DeclareMathOperator\Fix{Fix}

\DeclareMathOperator\reg{reg}

\newcommand{\Fq}{\mathbb{F}_{q}}
\newcommand{\Fqr}{\mathbb{F}_{q^{r}}}

\newcommand{\Ql}{\mathbb{Q}_\ell} % For l-adic numbers

\newcommand{\ff}{\mathrm{FF}}
\newcommand{\prob}{\text{Prob}}

\newcommand{\ordq}[1]{\ord_{#1}(q)}

\newcommand\la{\lambda}

\usepackage[scale=.76]{geometry}

\title[Heights, Entropy, and Zeta on Weighted Hypersurfaces over $\F_q$]{Stacky Heights, Entropy, and Zeta Functions of Weighted Hypersurfaces over Finite Fields}



\author{S. Salami}
\address{Institute of Mathematics and Statistics, Rio de Janeiro State University, Maracanã, Rio de Janeiro, 20950-000, RJ, Brazil}
\email{Sajad.salami@ime.uerj.br}

\author{T. Shaska}
\address{Department of Mathematics and Statistics,
Oakland University,  Rochester, MI, 48309.} \email{shaska@oakland.edu}

\begin{document}

\begin{abstract}

Let \(X\subset \mathbb P^n_{\mathbf w}\) be a weighted hypersurface over \(\mathbb F_q\), with associated stack \(\mathfrak X\). For \(\mathbf x\in X(\mathbb F_{q^r})\), let \(g_{\mathbf x}(r)=\gcd(k_{S(\mathbf x)},q^r-1)\), where \(k_S=\gcd(w_i:i\in S)\), and define \(\Theta_r(s)=\sum_{\mathbf x}g_{\mathbf x}(r)^{1-s}\) and \(Z_H^{\mathrm{can}}(\mathfrak X,s;t)=\exp(\sum_{r\ge1}\Theta_r(s)t^r/r)\). We prove the finite sector factorization
\[
Z_H^{\mathrm{can}}(\mathfrak X,s;t)=\prod_{e\in E_{\mathbf w}} Z(X^{(e)}_{\mathbb F_{q^{o_e}}},t^{o_e})^{J_{1-s}(e)/o_e},
\]
where \(X^{(e)}\) is the closed isotropy sector, \(o_e=\operatorname{ord}_e(q)\) for \(e>1\) and \(o_1=1\), and \(J_{1-s}\) is the Jordan totient function. Hence the specializations \(s=1,0,-1,-2,\ldots\) are rational in \(t\): \(s=1\) gives the Hasse--Weil zeta function of the coarse space, \(s=0\) the twist zeta function, and \(s=1-k\) the masses of the \(k\)-fold inertia stack. The factorization yields a functional equation when the nonempty sectors are self-dual and equidimensional; if they are also pure and geometrically irreducible, equidimensionality is necessary for every integral \(s\le0\). For weighted diagonal hypersurfaces of degree \(D\) with \(p\nmid D\), a finite Fermat cover gives purity and self-duality. We also prove a Lang--Weil asymptotic for isotropy entropy, with periodic leading term governed by isotropy periods and Frobenius orbits, and identify the function-field height with the stable stack height of Ellenberg--Satriano--Zureick-Brown, while the degree-based height remains a coordinate-complexity statistic.

\end{abstract}

\maketitle

%\tableofcontents

\section{Introduction}

Weighted projective hypersurfaces over finite fields combine ordinary point-counting with nontrivial isotropy coming from the weighted \(\G_m\)-action.  The coarse variety remembers the underlying projective geometry, whereas the quotient stack retains the stabilizers of coordinate strata.  The purpose of this paper is to make that distinction arithmetic: we package the rational isotropy of a weighted hypersurface into a family of zeta functions and identify its contributions sector by sector.

The starting point is \cite{paper-1}.  There the coarse rational-point count, the stack mass, and the number of \(\Fq^\times\)-orbits on nonzero representatives were separated for weighted projective stacks.  The first two are insensitive to a common reduction of the weights, while the orbit count records the \(\Fq\)-isomorphism classes of stack objects.  The ambient twist zeta function was shown to be rational and its functional equation was determined; weighted diagonal hypersurfaces were also treated explicitly in a split regime.  The present paper does not reprove those results.  Its new input is a closed-sector decomposition valid for an arbitrary weighted hypersurface and for the full family of isotropy moments \(\Theta_r(s)\).

Let \(\mathfrak X\subset\cP_{\w}\) be the stack associated with a weighted hypersurface \(X\).  A coarse point \(\x\in X(\Fqr)\) of support \(S(\x)\) has
\[
g_{\x}(r)=\gcd(k_{S(\x)},q^r-1)
\]
isomorphism classes of lifts to \(\mathfrak X(\Fqr)\), and each such lift has the same number of rational automorphisms.  We therefore introduce
\[
\Theta_r(s)=\sum_{\x\in X(\Fqr)}g_{\x}(r)^{1-s},\qquad
Z_H^{\can}(\mathfrak X,s;t)
=\exp\Bigl(\sum_{r\ge1}\Theta_r(s)\frac{t^r}{r}\Bigr).
\]
Theorem~\ref{thm:main} expresses this series as a finite product of ordinary Hasse--Weil zeta functions of closed isotropy sectors.  This is the structural result of the paper.  It implies, without a case-by-case orbit calculation, that \(s=1\) gives the zeta function of the coarse space, \(s=0\) gives the twist zeta function, and \(s=1-k\) records the mass of the \(k\)-fold inertia stack.  The factorization also isolates the geometric obstruction to a single functional equation: sectors of different dimensions have different centers of duality.  We prove sufficiency under self-duality, and obtain necessity as well when the sectors are pure and geometrically irreducible.

The height and entropy constructions provide two complementary interpretations of these coefficients.  The probabilistic height is the information content of a stack object for the normalized groupoid mass, and its mean is the corresponding groupoid entropy.  The entropy of the isotropy stratification is controlled by the non-generic rational isotropy locus; a Lang--Weil argument gives its leading decay and shows that the periodic coefficient depends not only on the periodic functions \(g_S(r)\) but also on Frobenius orbits of the geometric components.  This Frobenius contribution is essential even in elementary diagonal examples.

Two other height constructions are kept conceptually separate.  The function-field height over \(\Fq(t)\) is an arithmetic height in the usual global-field sense and is identified in Section~\ref{sec:6} with the stable height on the weighted projective stack in the sense of Ellenberg--Satriano--Zureick-Brown \cite{ellenberg}.  By contrast, the degree-based height is a coordinate-dependent finite-field complexity statistic with a Northcott property; it is useful for stratification but is not asserted to be an intrinsic Weil height.  Under the common reduction \(\w\mapsto\w/d\), the normalized stacky versions of these heights are invariant in the precise sense proved below, while the distribution of twists changes by the resulting \(\mu_d\)-gerbe.

The paper is organized as follows.  \Cref{sec:2} fixes notation.  \Cref{sec:3,sec:4} introduce the three height functions and their entropy profiles.  \Cref{sec:5} gives the groupoid and inertia interpretations, and \cref{sec:6} compares the function-field height with the stack height of \cite{ellenberg}.  \Cref{sec:7,sec:8} record the power-sum and height-transform identities used later.  The sector factorization, its entropy consequences, and the functional-equation criteria are proved in \cref{sec:9}.

%*********************************************
\section{Preliminaries}\label{sec:2}

Let \(k\) be a field and \(\w = (w_0, \dots, w_n)\) a vector of positive integers. The \textbf{weighted projective space} \(\bP_\w^n\) is the variety over \(k\) whose geometric points are the quotient of \(\bar k^{n+1} \setminus \{0\}\) under the action
\[
(x_0, \dots, x_n) \sim (\la^{w_0} x_0, \dots, \la^{w_n} x_n)
\]
 for \(\la \in \bar k^\times\),
 denoted by points \(\x = [x_0 : \dots : x_n]\). This space is \textbf{well-formed} if \(\gcd(w_0, \dots, \hat{w_i}, \dots, w_n) = 1\) for each \(i\), and \textbf{reduced}  if \(\gcd(w_0, \dots, w_n) = 1\). Weighted varieties are subschemes of \(\bP_\w^n\) defined by weighted homogeneous polynomials.
In characteristic zero, weighted heights generalize the Weil height to weighted projective spaces. In \cite{b-g-sh}, the \textbf{weighted greatest common divisor}  for a tuple \((a_0, \dots, a_n) \in \Z^{n+1}\) and weights \(\w\) is
\begin{equation}
\wgcd_\w    (a_0, \dots, a_n) := \max  \{ k \in \Z_{>0} \mid k^{w_i} \mid a_i \ \forall i \},
\end{equation}
 enabling coordinate normalization. The weighted height \(H_\w(P)\) for \(P \in \bP_\w^n(\bar{\Q})\) is defined as the product of local maxima of normalized absolute values over places, satisfying Northcott finiteness and subadditivity.

In \cite{s-sh}, local and global weighted heights are extended using Cartier and Weil divisors. The global height is
\begin{equation}
H_\w(P) := \prod_v H_{\w,v}(P)^{\deg(v)/[K:\Q]},
\end{equation}
 where \(H_{\w,v}\) are local heights at places \(v\) of a number field \(K\). These heights apply to Diophantine geometry, bounding rational points and analogs of Vojta's conjecture extended to smooth weighted varieties in \cite{2023-1}.

Put \(m = \lcm(w_0,\dots,w_n)\). The Veronese-type map
\(\phi_m: \bP_\w^n \to \bP^n\),
\[
 [x_0 : \dots : x_n] \mapsto [x_0^{m/w_0} : \dots : x_n^{m/w_n}],
\]
 is a finite morphism which relates weighted and classical points, but over a number field it is not surjective on rational points, inducing arithmetic sparsity: in \cite{sh-94} the asymptotics for points of bounded weighted height over number fields carry a sparsity factor coming from nontrivial torsors under groups of roots of unity.

Over finite fields \(\F_q\), the valuation-theoretic height is no longer available. The torsors that occur in the weighted quotient instead appear in the count of rational representatives. We use the following notation and facts from \cite{paper-1}. Let \(T=\{0,\dots,n\}\); for \(\x\) the support is \(S(\x)=\{i : x_i\neq 0\}\), and for \(\emptyset\neq S\subseteq T\) we put \(k_S=\gcd\{w_i : i\in S\}\), \(d=k_T=\gcd(\w)\), and
\[
g_S(r) := \gcd(k_S, q^r-1) = \#\mu_{k_S}(\Fqr) = \# H^1(\Fqr,\mu_{k_S}).
\]
The sequence \(r\mapsto g_S(r)\) is periodic, with period dividing the multiplicative order of \(q\) modulo the prime-to-\(p\) part of \(k_S\). The weighted projective stack is \(\cP_\w=[(\bA^{n+1}\setminus\{0\})/\G_m]\), with coarse space \(\bP_\w^n\). A coarse \(\Fq\)-point always admits an \(\Fq\)-representative \cite[Lem.~2.1]{paper-1}, the coarse count \(|\bP^n_\w(\Fq)| = (q^{n+1}-1)/(q-1)\) is weight-independent \cite[Cor.~2.2]{paper-1} (see also \cite{AubryPerret2025}), and the isomorphism classes of \(\cP_\w(\Fq)\) are the \(\Fq^\times\)-orbits on \(\Fq^{n+1}\setminus\{0\}\), counted by the orbit count \(A_\w(q)\) \cite[Thm.~3.4, Thm.~3.8]{paper-1}. If \(\w'=\w/d\), then \(\bP^n_\w\iso\bP^n_{\w'}\) as coarse varieties, while \(\cP_\w\to\cP_{\w'}\) is a \(\mu_d\)-gerbe; a point of support \(S\) has stabilizer \(\mu_{k_S}\) in \(\cP_\w\) and \(\mu_{k_S/d}\) in \(\cP_{\w'}\) \cite[Prop.~4.1]{paper-1}. Following \cite{paper-1} we call \(\w\mapsto\w/d\) the \textbf{reduction} of the weight vector.

Three notions of point must be kept apart over a field \(k\): the \(k\)-points of the coarse variety \(\bP^n_\w\); the objects of the groupoid \(\cP_\w(k)\); and the \(k^\times\)-orbits on \(k^{n+1}\setminus\{0\}\). The last two agree over every field, and the first is a quotient of them. We record this for an arbitrary field, since we need it for \(k=\Fqr\) and for \(k=\Fq(t)\).

\begin{lem}\label{lem:lifting}
Let \(k\) be any field.
\begin{enumerate}[label=(\roman*)]
\item The isomorphism classes of \(\cP_\w(k)\) are the \(k^\times\)-orbits on \(k^{n+1}\setminus\{0\}\) for the weighted action.
\item Every \(k\)-point of the coarse variety \(\bP^n_\w\) has a representative in \(k^{n+1}\setminus\{0\}\).
\item The orbits lying over a coarse \(k\)-point of support \(S\) form a torsor under \(H^1(k,\mu_{k_S})\simeq k^\times/(k^\times)^{k_S}\): if \(\tilde\x\) is one representative, the others are, up to the action of \(k^\times\), the vectors \((a^{w_i/k_S}\tilde x_i)_i\) with \(a\in k^\times\).
\end{enumerate}
The same holds for a closed substack \(\mathfrak X\subset\cP_\w\) defined by weighted homogeneous equations and its coarse space \(X\).
\end{lem}

\begin{proof}
An object of \(\cP_\w(k)\) is a \(\G_m\)-torsor over \(\spec k\) with an equivariant map to \(\bA^{n+1}\setminus\{0\}\). Hilbert's Theorem 90 gives \(H^1(k,\G_m)=0\), proving (i). Parts (ii) and (iii) are \cite[Lem.~2.1, Prop.~3.7]{paper-1}. Their proofs use only this vanishing and the fppf Kummer sequence for \(\mu_{k_S}\), so they apply over any field. For a substack, the cone is stable under the weighted action; hence every representative of a point of \(X\) lies on it.
\end{proof}

Over a finite field the torsor in (iii) is finite; over \(\Fq(t)\) it is infinite when \(k_S>1\), and trivial when \(k_S=1\). Heights attached to the coarse point and heights attached to the stack object must therefore be distinguished over the function field, which we do in \cref{sec:6}.

Let now \(X=V(f)\subset\bP^n_\w\) be a weighted hypersurface over \(\Fq\), with affine cone \(\hat X=\{f=0\}\subset\bA^{n+1}\), and let
\[
\mathfrak{X} = [(\hat X\setminus\{0\})/\G_m] \subset \cP_\w
\]
be the associated closed substack, with coarse space \(X\). We write \(X(\Fqr)\) for the \(\Fqr\)-points of the coarse variety, equivalently the geometric points of \(X\) fixed by the \(q^r\)-Frobenius, and \(\pi_0(\mathfrak{X}(\Fqr))\) for the set of isomorphism classes of the groupoid \(\mathfrak{X}(\Fqr)\), which we call \textbf{twists}. The following lemma collects what we need; it is \cite[Prop.~3.7, Prop.~3.9, Prop.~8.1]{paper-1} read on the cone of \(X\).

\begin{lem}\label{lem:twists}
Let \(\x\in X(\Fqr)\) have support \(S\). Then:
\begin{enumerate}[label=(\roman*)]
\item \(\x\) has a representative in \(\hat X(\Fqr)\setminus\{0\}\), and every \(\Fqr\)-representative of \(\x\) lies on \(\hat X\);
\item the \(\Fqr^\times\)-orbits of \(\Fqr\)-representatives of \(\x\), that is, the twists lying over \(\x\), form a torsor under \(H^1(\Fqr,\mu_{k_S})\); there are \(g_S(r)\) of them, each orbit has \((q^r-1)/g_S(r)\) elements, and each twist \(\xi\) has \(\Aut(\xi)(\Fqr)=\mu_{k_S}(\Fqr)\), of order \(g_S(r)\);
\item consequently
\[
\#\bigl(\hat X(\Fqr)\setminus\{0\}\bigr) = (q^r-1)\,|X(\Fqr)|,
\qquad
\#\pi_0(\mathfrak{X}(\Fqr)) = \sum_{\x\in X(\Fqr)} g_{S(\x)}(r).
\]
\end{enumerate}
\end{lem}

\begin{proof}
Parts (i) and (ii) are \cref{lem:lifting} for \(k=\Fqr\), where \(\#H^1(\Fqr,\mu_{k_S})=g_S(r)\) and the stabilizer of a vector of support \(S\) in \(\Fqr^\times\) is \(\mu_{k_S}(\Fqr)\). Part (iii) follows by summing over \(X(\Fqr)\).
\end{proof}

Two integers are attached to a point \(\x\) of support \(S\) and must not be confused: the order \(k_S\) of the geometric stabilizer group scheme \(\mu_{k_S}\), which we write \(|\Aut(\x)|\) and which does not depend on \(r\); and the number \(g_\x(r)=g_{S}(r)=\gcd(k_S,q^r-1)\) of its \(\Fqr\)-rational automorphisms, which is also the number of its twists. Always \(g_\x(r)\mid k_S\). The second count in (iii) is the orbit count \(A_X(q^r)\) of \cite[Sec.~8]{paper-1}. Throughout, \(\log\) denotes the natural logarithm.

%******************************************************************
\section{Heights on Weighted Varieties over Finite Fields}\label{sec:3}

Weighted heights over global fields use valuations and therefore have no direct analogue over the finite field \(\F_q\) itself.  We shall use three different functions associated with a weighted hypersurface \(X\subseteq\bP^n_{\w}\) over \(\F_q\):
\begin{enumerate}[label=\roman*)]
\item a degree-based coordinate-complexity function,
\item the valuation-theoretic height after base change to \(\F_q(t)\), and
\item a probabilistic height obtained from weighted orbit sizes.
\end{enumerate}

Each family induces a stratification of a point set attached to \(X\) and will be used later to define entropy and height zeta functions. The first lives on the geometric points \(X(\overline{\F}_q)\), the second on the points of \(X\) over the function field, and the third on the twists over \(X(\Fqr)\).

%*************************************************************************************************

\subsection{A degree-based complexity function}

The first construction uses extension degrees of coordinates.  It is useful for finite-field stratifications, but it is not an intrinsic Weil height; see \cref{rem:deg-status}.

For a geometric point \(\x = [x_0 : \cdots : x_n] \in X(\overline{\F}_q)\), let
\[
r_\x := \min\{ r \geq 1 \mid \x\text{ admits a representative with all coordinates in }\F_{q^r}\}
\]
be the degree of the minimal field of definition of \(\x\). For \(a\in\overline{\F}_q\), let \(\deg(a)=[\F_q(a):\F_q]\) denote the degree of its minimal polynomial over \(\F_q\) (so that \(\deg(a)=1\) if \(a\in\F_q\), and in particular \(\deg(0)=1\)).

\begin{defn}
The \textbf{degree-based weighted height} is
\[
H_{\w,q}^{\deg}(\x) := r_\x \cdot \min_{\lambda \in \overline{\F}_q^\times} \max_{0\leq i\leq n} \frac{\deg(\lambda^{w_i} x_i)}{w_i}.
\]
Since the denominators \(w_i\) are fixed, the inner maximum takes values in a discrete subset of \(\frac1m\Z_{>0}\), where \(m=\lcm(\w)\); hence the minimum is attained.
\end{defn}

This definition is independent of the choice of representative: it minimises over the full weighted orbit in the algebraic closure. If \((x'_0,\dots,x'_n) = (\mu^{w_0}x_0,\dots,\mu^{w_n}x_n)\) for \(\mu\in\overline{\F}_q^\times\), the set of scalars \(\lambda\) simply shifts, so the minimum is unchanged. The weights enter non-trivially through the denominators \(w_i\).

\begin{rem}
We set \(\deg(0)=1\).  This convention keeps every term in the maximum positive and will be used throughout this subsection.
\end{rem}

\begin{prop}\label{Pp1}
The function \(H_{\w,q}^{\deg}\) satisfies:
\begin{enumerate}[label=(\roman*)]
\item (Finiteness) For any \(B>0\), the set \(\{\x\in X(\overline{\F}_q)\mid H_{\w,q}^{\deg}(\x)\leq B\}\) is finite.
\item (Weighted scaling) If \(\x\sim\lambda\cdot\x\), then \(H_{\w,q}^{\deg}(\x)=H_{\w,q}^{\deg}(\lambda\cdot\x)\).
\item (Functoriality of field of definition) If \(\phi:X\to Y\) is a morphism of weighted varieties defined over \(\F_q\), then \(r_{\phi(\x)}\leq r_\x\).
\end{enumerate}
\end{prop}

\begin{rem}
Part (iii) establishes functoriality only for the extension-degree factor \(r_\x\). It does not automatically imply functoriality of the full height \(H_{\w,q}^{\deg}\), because the inner min-max term may behave differently under a general morphism \(\phi\).
\end{rem}

\begin{proof}
(i) By definition,
\[
H_{\w,q}^{\deg}(\x)=r_\x\cdot\min_{\lambda\in\overline{\F}_q^\times}\max_i\frac{\deg(\lambda^{w_i}x_i)}{w_i}.
\]
The inner min-max term is always at least \(1/\min_i w_i>0\) (since \(\deg(\cdot)\geq1\)). Therefore \(H_{\w,q}^{\deg}(\x)\leq B\) forces
\[
r_\x\leq B\cdot\min_i w_i.
\]
There are only finitely many positive integers \(r\leq B\cdot\min_i w_i\). For each such \(r\), every geometric point \(\x\) with \(r_\x=r\) admits a representative defined over \(\F_{q^r}\), hence determines an \(\F_{q^r}\)-rational point of \(X\).
Since each \(X(\F_{q^r})\) is finite, only finitely many geometric points can occur. Therefore the total set of geometric points of bounded height is finite.

(ii) Let \(\x'=\mu\cdot\x\) for some \(\mu\in\overline{\F}_q^\times\). Then any representative of \(\x'\) is of the form \((\mu^{w_i}x_i)\). The weighted orbit of \(\x'\) consists of points
\[
\lambda^{w_i}(\mu^{w_i}x_i)=(\lambda\mu)^{w_i}x_i,
\]
so as \(\lambda\) runs through \(\overline{\F}_q^\times\), the set \(\{\lambda\mu\mid\lambda\in\overline{\F}_q^\times\}\) is again the full group \(\overline{\F}_q^\times\). Therefore the minimum taken over the orbit of \(\x'\) is identical to the minimum taken over the orbit of \(\x\). Hence \(H_{\w,q}^{\deg}(\x')=H_{\w,q}^{\deg}(\x)\).

(iii) Suppose \(\phi\) is given by polynomials with coefficients in \(\F_q\). Let \(\x\) be a geometric point with minimal field of definition \(\F_{q^{r_\x}}\), i.e., \(\x\) admits a representative with coordinates in \(\F_{q^{r_\x}}\). Since the defining equations of \(\phi\) have coefficients in \(\F_q\), applying \(\phi\) to this representative yields a representative of \(\phi(\x)\) whose coordinates also lie in \(\F_{q^{r_\x}}\). Therefore the minimal extension degree of \(\phi(\x)\) satisfies \(r_{\phi(\x)}\leq r_\x\).
\end{proof}

Since the degree-based height is invariant under the weighted action (Prop.~\ref{Pp1}(ii)), the value \(H_{\w,q}^{\deg}(\x)\) is well-defined on geometric points of \(X\).

\begin{rem}
The stronger coordinate-degree subadditivity
\[
\deg(f_j(x))\leq d_j \max_i\deg(x_i)
\]
does not hold in general over finite fields, so we state only the safe functoriality of the extension degree \(r_\x\).
\end{rem}

\begin{rem}[Status of the degree-based height]\label{rem:deg-status}
We call \(H^{\deg}_{\w,q}\) a height because it satisfies the Northcott property and is formally modelled on the weighted height of \cite{b-g-sh}, with extension degrees in place of valuations. It should be understood, however, as a complexity statistic attached to the chosen weighted coordinates, not as an intrinsic Weil height.  We prove invariance under the weighted \(\G_m\)-action and the common reduction of the weights, but we do not claim invariance under arbitrary automorphisms of the weighted projective space, nor do we attach a line-bundle interpretation to it.  The intrinsic quantity entering its definition is \(r_\x\), the degree of the residue field of the coarse point. None of the main results of \cref{sec:9} depends on \(H^{\deg}_{\w,q}\).
\end{rem}

As a direct consequence of the definition, the degree-based height behaves under the reduction of the weight vector as follows.

\begin{cor}
\label{COR1}
Let \(d=\gcd(\w)\) and let \(\w'=\w/d\). Let \(\x\in X\subset\mathbb{P}^n_{\w}\) and \(\x'\in X'\subset\mathbb{P}^n_{\w'}\) be corresponding points on the isomorphic hypersurfaces (with the reduced weight vector \(\w'\)). Then
\[
H_{\w',q}^{\deg}(\x')=d\cdot H_{\w,q}^{\deg}(\x).
\]
\end{cor}

\begin{proof}
Since \(\x\) and \(\x'\) are represented by the same homogeneous coordinates, their minimal fields of definition coincide, so \(r_{\x'}=r_\x\). Therefore,
\[
H_{\w',q}^{\deg}(\x')=r_\x\cdot\min_{\mu\in\overline{\F}_q^\times}\max_i\frac{\deg(\mu^{w_i/d}x_i)}{w_i/d}.
\]
Perform the change of variables \(\mu=\lambda^d\) with \(\lambda\in\overline{\F}_q^\times\). Then
\[
\mu^{w_i/d}=(\lambda^d)^{w_i/d}=\lambda^{w_i},
\]
so
\[
\deg(\mu^{w_i/d}x_i)=\deg(\lambda^{w_i}x_i).
\]
Hence
\[
\max_i\frac{\deg(\mu^{w_i/d}x_i)}{w_i/d}=d\cdot\max_i\frac{\deg(\lambda^{w_i}x_i)}{w_i}.
\]
It remains to check that the substitution \(\mu=\lambda^d\) runs over the entire group \(\overline{\F}_q^\times\) as \(\lambda\) does. Since \(\overline{\F}_q\) is algebraically closed, the map
\[
\overline{\F}_q^\times\longrightarrow\overline{\F}_q^\times,\qquad\lambda\mapsto\lambda^d
\]
is surjective. Therefore the set of values attained by the min-max expression for the primed weights is exactly \(d\) times the corresponding expression for the original weights.

Multiplying by the common factor \(r_\x\) yields
\[
H_{\w',q}^{\deg}(\x')=d\cdot H_{\w,q}^{\deg}(\x).
\]
\end{proof}

The next result provides simple but useful bounds for the degree-based weighted height.

\begin{prop}\label{prop:deg-bounds}
For any geometric point \(\x \in X(\overline{\F}_q)\),
\[
\frac{r_\x}{\min_i w_i} \;\leq\; H_{\w,q}^{\deg}(\x) \;\leq\; \frac{r_\x^{2}}{\min_i w_i}.
\]
In particular \(H_{\w,q}^{\deg}(\x)=1/\min_i w_i\) for every \(\x\in X(\Fq)\).
\end{prop}

\begin{proof}
For any fixed \(\lambda\in\overline{\F}_q^\times\), each coordinate satisfies \(\deg(\lambda^{w_i}x_i)\geq 1\) (adopting the convention \(\deg(0)=1\)). Hence
\[
\max_i \frac{\deg(\lambda^{w_i}x_i)}{w_i} \geq \max_i\frac{1}{w_i} = \frac{1}{\min_i w_i}.
\]
The minimum over \(\lambda\) is still bounded below by the same quantity. Multiplying by \(r_\x\) yields the lower bound. For the upper bound take \(\lambda=1\) and a representative with coordinates in \(\F_{q^{r_\x}}\); then \(\deg(x_i)\leq r_\x\) for all \(i\), so the inner term is at most \(r_\x/\min_i w_i\). If \(r_\x=1\) the two bounds agree.
\end{proof}

%*******************************************************************************
\subsection{Heights over Function Fields}

We next pass to the global function field \(K=\F_q(t)\), where the usual valuation-theoretic definition of height is available. For a place \(v\) of \(K\) we normalize \(|a|_v = q^{-\deg(v)\ord_v(a)}\), so that the product formula \(\sum_{v}\log|a|_v=0\) holds for \(a\in K^\times\), and we put \(\log|0|_v=-\infty\).

\begin{defn}
Let \(X_K = X\times_{\Fq} K\) be the base change of \(X \subseteq \mathbb{P}_{\w}^n\) to \(K = \F_q(t)\).
Let \(\x \in X_K(K)\) be a \(K\)-point of the coarse variety. By \cref{lem:lifting} applied to \(k=K\), it has a weighted-homogeneous representative \(\tilde{\x} = (\tilde{x}_0 , \dots , \tilde{x}_n)\) in \(K^{n+1}\setminus\{0\}\). The \textbf{function field weighted height} of \(\x\) is
\begin{equation}\label{height-FF}
H_{\w,q}^{\ff}(\x) = \sum_{v \in M_K} \max_{0\le i\le n} \, \frac{\log |\tilde{x}_i|_v}{w_i},
\end{equation}
where \(v\) runs over all places of \(K\). For a finite extension \(L/K\) and \(\x\in X_K(L)\) the height is defined by the same formula, with \(v\) running over the places of \(L\) and \(|a|_v=|\kappa(v)|^{-\ord_v(a)}\) for the residue field \(\kappa(v)\) at \(v\); this is the absolute normalization, in which \(H^{\ff}\) scales by \([L:K]\) under base change and not by \(1\).
\end{defn}

This is the logarithmic weighted height of \cite{b-g-sh, s-sh} over the global field \(K\). It is a function on the coarse points: by the next lemma it takes the same value on all the twists of \cref{lem:lifting}(iii) lying over \(\x\). In general this does not follow from \(H^{\ff}\) alone, since a coarse point with nontrivial stabilizer has infinitely many twists; the next lemma proves the independence directly, from the product formula.

\begin{lem}
For any point \(\x \in X_K(K)\), the function field weighted height \(H_{\w,q}^{\ff}(\x)\) is well-defined and independent of the choice of weighted-homogeneous representative \(\tilde{\x} \in K^{n+1}\setminus\{0\}\). Moreover \(m\, H_{\w,q}^{\ff}(\x)\) is the logarithmic Weil height of \(\phi_m(\x)\in\bP^n(K)\), where \(m=\lcm(\w)\); in particular \(H_{\w,q}^{\ff}\) takes values in \(\frac{\log q}{m}\,\Z_{\geq 0}\).
\end{lem}

\begin{proof}
At each place \(v\) one has
\[
\max_i \frac{\log|\tilde x_i|_v}{w_i} = \frac1m \max_i \log\bigl|\tilde x_i^{\,m/w_i}\bigr|_v ,
\]
so \(m\,H_{\w,q}^{\ff}(\x)\) is the Weil height of the vector \(\phi_m(\tilde\x)\). If \(\tilde{\y}\) is a second representative in \(K^{n+1}\), then \(\tilde{y}_i = f^{w_i} \tilde{x}_i\) for some \(f \in \bar K^\times\) with \(f^{w_i}\in K\) for all \(i\) in the support, hence \(f^{m}\in K^\times\) and \(\phi_m(\tilde\y)=f^m\phi_m(\tilde\x)\). By the product formula
\[
\sum_{v \in M_K} \log |f^m|_v = 0
\]
the two vectors have the same height, which is the height of the point \(\phi_m(\x)\in\bP^n(K)\). The Weil height on \(\bP^n(K)\) is a non-negative integer multiple of \(\log q\).
\end{proof}

The next theorem records Northcott finiteness and identifies the points of height zero.

\begin{thm}
\label{thm:FF-finiteness}
Let $X $ be a weighted hypersurface in $\bP_{\w}^n$ over \(\Fq\) and \(K=\F_q(t)\).
\begin{enumerate}[label=(\roman*)]
\item For any  constant \(B>0\), the set
\(
\{\x \in X_K(K) \mid H_{\w,q}^{\ff}(\x) \leq B\}
\)
is finite.
\item \(H_{\w,q}^{\ff}(\x)\geq 0\), with equality if and only if \(\x\in X(\Fq)\subset X_K(K)\).
\end{enumerate}
\end{thm}

\begin{proof}
(i) The Veronese map \(\phi_m\colon \bP^n_\w\to\bP^n\) is finite, and \(h(\phi_m(\x)) = m\,H_{\w,q}^{\ff}(\x)\) by the lemma. By the Northcott property for \(\bP^n\) over the global function field \(K\), whose constant field is finite, there are finitely many points of \(\bP^n(K)\) of height at most \(mB\); each has finitely many preimages.

(ii) Non-negativity is that of the Weil height. If \(H_{\w,q}^{\ff}(\x)=0\), then \(\phi_m(\x)\) is a point of height zero in \(\bP^n(K)\), that is, a point of \(\bP^n(\Fq)\). Its fibre \(Z\) under \(\phi_m\) is a finite \(\Fq\)-scheme, and \(Z(K)=Z(\Fq)\) since \(\Fq\) is algebraically closed in \(K\). Conversely a point with coordinates in \(\Fq\) has \(|\tilde x_i|_v\in\{0,1\}\) for all \(v\).
\end{proof}

\begin{rem}
See e.g.\ Bombieri--Gubler \cite{h-silv} \S2.4 for the Northcott property over function fields. Applied over \(\Fqr(t)\), part (ii) says that the \(\Fqr\)-points of \(X\) are precisely the points of function-field height zero. Thus \(H_{\w,q}^{\ff}\) detects nonconstant points over the function field.
\end{rem}

The function-field height is functorial in the following elementary sense.

\begin{prop}
\label{PP2}
Let \(\phi: X \to Y\) be a morphism of weighted varieties defined over \(\F_q\), induced by
\[
\phi([x_0:\dots:x_n]) = [f_0(x):\dots:f_m(x)],
\]
where each \(f_j\in\Fq[x_0,\dots,x_n]\) is weighted homogeneous of degree \(e w'_j\) for a common integer \(e\), with target \(Y\subset \mathbb{P}_{\w'}^m\). Then for every \(\x \in X_K(K)\),
\[
H_{\w',q}^{\ff}(\phi(\x)) \leq e \cdot H_{\w,q}^{\ff}(\x),
\]
with equality if the \(f_j\) have no common zero on \(\hat X\setminus\{0\}\).
\end{prop}

\begin{proof}
Put \(M_v=\max_i|\tilde x_i|_v^{1/w_i}\). A monomial \(\prod x_i^{a_i}\) of weighted degree \(\sum a_iw_i = ew'_j\) satisfies \(\prod|\tilde x_i|_v^{a_i}\le M_v^{\,ew'_j}\). The coefficients of \(f_j\) lie in \(\Fq\) and have absolute value at most \(1\) at every place, so the ultrametric inequality gives \(|f_j(\tilde\x)|_v^{1/w'_j}\le M_v^{\,e}\) for all \(v\) and \(j\). Taking logarithms and summing over \(v\) gives the upper bound. For the reverse inequality, if the \(f_j\) have no common zero on \(\hat X\setminus\{0\}\), the homogeneous Nullstellensatz in the graded coordinate ring of \(\hat X\) gives, for each \(i\), an integer \(N_i\) and weighted-homogeneous \(g_{ij}\) of degree \(w_iN_i-ew'_j\) such that \(x_i^{N_i}=\sum_j g_{ij}f_j\) on \(\hat X\). Put \(M'_v=\max_j|f_j(\tilde\x)|_v^{1/w'_j}\) and choose \(i\) with \(|\tilde x_i|_v^{1/w_i}=M_v\). Then
\[
M_v^{\,w_iN_i} \le \max_j M_v^{\,w_iN_i-ew'_j}\,(M'_v)^{w'_j},
\]
so \(M_v^{\,e}\le M'_v\). No constant appears because all coefficients are constants of \(K\).
\end{proof}

Under reduction of the weights the function-field height scales by the common factor.

\begin{cor}
\label{COR2}
Let \(\x \in X_K(K)\) and \(\x' \in X'_K(K)\) be corresponding points on isomorphic hypersurfaces under the reduction \(\w' = \w/d\). Then
\[
H_{\w',q}^{\ff}(\x') = d \cdot H_{\w,q}^{\ff}(\x).
\]
\end{cor}

\begin{proof}
The two points have the same representatives, and at every place \[\frac{\log|\tilde x_i|_v}{w_i/d} = d\cdot\frac{\log|\tilde x_i|_v}{w_i}.\]
\end{proof}

In terms of line bundles, the gerbe \(\pi\colon\cP_\w\to\cP_{\w'}\) satisfies \(\pi^*\Oo_{\cP_{\w'}}(1)\simeq\Oo_{\cP_\w}(d)\), and \cref{COR2} is the corresponding identity of heights; we return to this in \cref{sec:6}.

%*************************************************************************************************
\subsection{Weighted Probabilistic Measures}

For the third construction we use the size of the weighted orbit on the affine cone.  The orbit cannot be taken in \(\bP^n_\w\), where \([\la^{w_0}x_0:\cdots:\la^{w_n}x_n]\) represents the same point as \(\x\).

\begin{defn}
For a point \(\x \in X(\F_{q^r})\) with representative \(\tilde\x\in\hat X(\Fqr)\setminus\{0\}\), the \textbf{probabilistic weighted height} is
\[
H_{\w,q}^{\prob}(\x)
=
-\log\left(
\frac{|\Orb_{\w}(\tilde\x)|}{|\hat X(\F_{q^r})\setminus\{0\}|}
\right),
\]
where \(\Orb_{\w}(\tilde\x)\) denotes the orbit of \(\tilde\x\) under the weighted action of
\(\F_{q^r}^{\times}\), namely
\[
\Orb_{\w}(\tilde\x)
=
\left\{
(\lambda^{w_0}x_0,\dots,\lambda^{w_n}x_n)
\;\middle|\;
\lambda\in\F_{q^r}^{\times}
\right\}\subset \hat X(\Fqr)\setminus\{0\}.
\]
\end{defn}

Thus \(H_{\w,q}^{\prob}(\x)\) is the negative logarithm of the probability that a uniformly chosen nonzero point of the cone lies in the orbit of \(\tilde\x\).

\begin{prop}\label{prop:prob-formula}
The function \(H_{\w,q}^{\prob}\) is independent of the choice of representative, and
\[
H_{\w,q}^{\prob}(\x) = \log |X(\Fqr)| + \log g_{\x}(r),
\qquad g_\x(r)=\gcd(k_{S(\x)},q^r-1).
\]
In particular
\[
\log|X(\Fqr)| \;\le\; H_{\w,q}^{\prob}(\x) \;\le\; \log|X(\Fqr)| + \log\max_i w_i ,
\]
with equality on the left if and only if \(\x\) has a single twist over \(\Fqr\).
\end{prop}

\begin{proof}
By \cref{lem:twists}, every orbit of a representative of \(\x\) has \((q^r-1)/g_\x(r)\) elements and the cone has \((q^r-1)|X(\Fqr)|\) nonzero points. The quotient is \(1/(g_\x(r)\,|X(\Fqr)|)\), which depends only on the support of \(\x\). The bounds follow from \(1\le g_\x(r)\le k_{S(\x)}\le\max_iw_i\).
\end{proof}

\begin{prop}
	\label{ppp3}
Let \(X \subseteq \mathbb{P}_\w^n\) and \(Y \subseteq \mathbb{P}_{\w'}^m\) be weighted hypersurfaces defined over \(\mathbb{F}_q\), and let \(\phi: X \to Y\) be a morphism of weighted varieties as in \cref{PP2}. For any point \(\x \in X(\mathbb{F}_{q^r})\),
\[
H_{\w',q}^{\prob}(\phi(\x)) - H_{\w,q}^{\prob}(\x) = \log \left( \frac{|Y(\mathbb{F}_{q^r})|}{|X(\mathbb{F}_{q^r})|} \right) + \log \left( \frac{g_{\phi(\x)}(r)}{g_{\x}(r)} \right).
\]
\end{prop}

\begin{proof}
Subtract the two expressions given by \cref{prop:prob-formula}.
\end{proof}

The following corollary describes the behaviour of the probabilistic weighted height under reduction of the weights.

\begin{cor}
\label{COR3}
Let \(X \subset \mathbb{P}_\w^n\) and \(X' \subset \mathbb{P}_{\w'}^n\) be the isomorphic hypersurfaces with weights related by \(\w' = \w/d\). For any point \(\x \in X(\mathbb{F}_{q^r})\) with corresponding point \(\x'\) in \(X'\) and support \(S\), we have
\[
H_{\w,q}^{\prob}(\x) - H_{\w',q}^{\prob}(\x') = \log \left( \frac{\gcd(k_S, q^r - 1)}{\gcd(k_S/d, q^r - 1)} \right),
\]
which lies between \(0\) and \(\log\gcd(d,q^r-1)\), and equals \(\log\gcd(d,q^r-1)\) whenever \(\gcd(d,k_S/d)=1\).
\end{cor}

\begin{proof}
The coarse varieties coincide, so \(|X(\Fqr)|=|X'(\Fqr)|\), and the stabilizers are \(\mu_{k_S}\) and \(\mu_{k_S/d}\). For the quotient of the two gcd's, put \(N=q^r-1\), which is prime to \(p\). For a prime \(\ell\) write \(a=v_\ell(k_S)\), \(e=v_\ell(d)\), \(n=v_\ell(N)\); then
\[
v_\ell\Bigl(\frac{\gcd(k_S,N)}{\gcd(k_S/d,N)}\Bigr)=\min(a,n)-\min(a-e,n)\in[0,\min(e,n)],
\]
so the quotient is an integer dividing \(\gcd(d,N)\), in any characteristic. When \(p\nmid d\) this is the index computed cohomologically in \cite[Cor.~4.3]{paper-1}. Finally \(\gcd(ab,N)=\gcd(a,N)\gcd(b,N)\) when \(\gcd(a,b)=1\), which gives equality when \(\gcd(d,k_S/d)=1\).
\end{proof}

Unlike \cref{COR1} and \cref{COR2}, the change is point-dependent: it is the local contribution of the \(\mu_d\)-gerbe at \(\x\).

%*******************************************************************
\section{Entropy and Point Distribution in Weighted Varieties}
\label{sec:4}

We attach Shannon entropy to the finite height profiles defined in \cref{sec:3}.  The resulting quantities measure only the distribution of rational points among the corresponding strata; no coding-theoretic interpretation is used below. Throughout, \(X\subset\bP^n_\w\) is a weighted hypersurface over \(\Fq\) and \(r\) is such that \(X(\Fqr)\neq\varnothing\).

We shall use the following running example.

\begin{exa}\label{exa:running}
Let \(\w=(1,1,2)\) and \(X=\{x_0=0\}\subset\bP_\w^2\), so that \(X\iso\bP(1,2)\) with coordinates \([x_1:x_2]\) and coarse space \(\bP^1\). Then \(|X(\Fqr)|=q^r+1\). The point \([0:0:1]\) has stabilizer \(\mu_2\), all other points have trivial stabilizer, and \(g_{[0:0:1]}(r)=\gcd(2,q^r-1)\). For \(q\) odd there are two twists over \([0:0:1]\), and \(q^r+2\) twists in all.
\end{exa}

%-----------------------------------------------

\subsection{Degree-Based Stratification and Entropy}
We begin with the degree-based weighted height \( H_{\w,q}^{\deg} \), which measures minimal extension degrees and coordinate complexity adjusted by weights.

For a fixed extension degree \( r \geq 1 \), the set of \(\F_{q^r}\)-rational points of \( X \subseteq \bP^n_{\w} \) stratified by this height is
\[
\cH_m^{\deg}(X, \F_{q^r}) := \{ \x \in X(\F_{q^r}) \mid H_{\w,q}^{\deg}(\x) = m \}.
\]
These sets partition \( X(\F_{q^r}) \):
\[
X(\F_{q^r}) = \bigsqcup_{m \in \mathbb{R}_{> 0}} \cH_m^{\deg}(X, \F_{q^r}).
\]
The height profile function is
\[
\mu_{q^r}^{\deg}(m) := \frac{|\cH_m^{\deg}(X, \F_{q^r})|}{|X(\F_{q^r})|},
\]
representing the empirical distribution. It is a probability mass function with finite support.

We define the \textbf{degree-based height entropy} by
\[
\Ent_{q^r}^{\deg}(X) := -\sum_{m} \mu_{q^r}^{\deg}(m) \log \mu_{q^r}^{\deg}(m),
\]
where the sum is over \( m \) with \(\mu_{q^r}^{\deg}(m) > 0\).

\begin{prop}\label{prop:entropy-bound}
	For every weighted hypersurface
\(
X\subset \bP_{\w}^{n}
\)
with \(X(\F_{q^r})\neq \varnothing\),
\[
0
\leq
\Ent_{q^r}^{\deg}(X)
\leq
\log |\Supp(\mu_{q^r}^{\deg})|.
\]
Moreover,
\[
|\Supp(\mu_{q^r}^{\deg})|
\leq
\sigma(r)\,\frac{w_0+\cdots+w_n}{\min_i w_i},
\]
where \(\sigma(r)\) is the sum of the divisors of \(r\), and \(\Ent_{q}^{\deg}(X)=0\).
\end{prop}

\begin{proof}
	The first inequality is the standard entropy bound for a probability measure
	on a finite set.

	It therefore suffices to estimate the number of possible values of
	\(H_{\w,q}^{\deg}\) on \(X(\Fqr)\).  For such a point \(r_{\x}\mid r\). The inner min-max term is attained at some \(\la\) and some index \(j\), so it has the form \(k/w_j\) with \(k\) a positive integer, and by \cref{prop:deg-bounds} it is at most \(r_\x/\min_iw_i\). Hence \(k\le r_\x w_j/\min_i w_i\), and for fixed \(r_\x\) and \(j\) there are at most \(r_\x w_j/\min_iw_i\) values. Summing over \(j\) and over the divisors \(r_\x\) of \(r\) gives the bound. For \(r=1\) the height is constant on \(X(\Fq)\) by \cref{prop:deg-bounds}.
\end{proof}

\begin{exa}
For the hypersurface of \cref{exa:running} and \( r = 1 \), every \(\x\in X(\Fq)\) has \( H_{\w,q}^{\deg}(\x) = 1 \), so \(\cH_1^{\deg}(X, \F_q) = X(\F_q)\) and \(\mu_q^{\deg}(1) = 1\), yielding \(\Ent_q^{\deg}(X) = 0\). For \( r = 2 \), a point \([0:1:x_2]\) with \( x_2 \in \F_{q^2} \setminus \F_q \) has \(r_\x=2\) and inner term \(\max(1,1,2/2)=1\), hence height \(2\). Thus
\[
\mu_{q^2}^{\deg}(1)=\frac{q+1}{q^2+1},\qquad \mu_{q^2}^{\deg}(2)=\frac{q^2-q}{q^2+1},
\]
which diversifies \(\mu_{q^2}^{\deg}\).
\end{exa}

As \( r \to \infty \), one may ask for the limiting distribution
\(
\mu^{\deg}(m) := \lim_{r \to \infty} \mu_{q^r}^{\deg}(m).
\)
It exists, but it is not a probability measure.

\begin{prop}[Escape of mass]\label{prop:escape}
Let \( X \) be geometrically irreducible of dimension \(n-1\ge1\). Then \(\mu^{\deg}(m)=0\) for every \(m\). More precisely, for every \(M>0\),
\[
\frac{\#\{\x\in X(\Fqr) : H^{\deg}_{\w,q}(\x)\le M\}}{|X(\Fqr)|} = O\bigl(q^{-r(n-1)}\bigr).
\]
\end{prop}

\begin{proof}
By \cref{prop:deg-bounds}, the inequality \(H^{\deg}_{\w,q}(\x)\le M\) forces \(r_\x\le M\min_iw_i\). Thus the numerator is bounded by \(\sum_{e\le M\min_i w_i}|X(\F_{q^e})|\), independently of \(r\). The denominator is \(q^{r(n-1)}(1+o(1))\) by Lang--Weil \cite{lang-weil}.
\end{proof}

\begin{rem}
The points of \(X(\Fqr)\) with \(r_\x<r\) number \(O(q^{r(n-1)/2})\), so almost all points are primitive, and for these \cref{prop:deg-bounds} gives \(r/\min_i w_i\le H^{\deg}_{\w,q}(\x)\le r^2/\min_iw_i\). The natural object as \(r\to\infty\) is therefore the distribution of the normalized height \(H^{\deg}_{\w,q}/r^2\) on \(X(\Fqr)\), supported in \((0,1/\min_iw_i]\).
\end{rem}

We do not pursue the limiting distribution of the normalized degree-based height in this paper.

%----------------------------------------------------------------------
\subsection{Function Field Stratification and Entropy}
We apply the function field weighted height \( H_{\w,q}^{\ff} \), which lives on \(X_K(K)\), \( K = \F_q(t) \). By \cref{thm:FF-finiteness} the points of \(X(\Fq)\) form the stratum of height zero, so this stratification continues the finite-field point set into the function field. Let \(m=\lcm(\w)\) and
\[
\cH_k^{\ff}(X) := \Bigl\{ \x \in X_K(K) \;\Big|\; H_{\w,q}^{\ff}(\x) = \frac{k\log q}{m} \Bigr\},\qquad N^{\ff}(k)=|\cH^{\ff}_k(X)|,
\]
which is finite for every \(k\ge0\), with \(\cH^{\ff}_0(X)=X(\Fq)\). For a bound \(B\ge0\) the height profile function is
\[
\mu_{B}^{\ff}(k) := \frac{N^{\ff}(k)}{\sum_{j\le B}N^{\ff}(j)},\qquad 0\le k\le B,
\]
and the \textbf{function field height entropy} is
\[
\Ent_{B}^{\ff}(X) := -\sum_{k\le B} \mu_{B}^{\ff}(k) \log \mu_{B}^{\ff}(k).
\]
The entropy satisfies \( 0 \leq \Ent_{B}^{\ff}(X) \leq \log (B+1) \).

When the strata grow exponentially, the profile concentrates near the top stratum and the entropy remains bounded, the limit being governed by the growth rate alone. Exponential growth is not automatic: if \(X\) is a curve of genus at least \(2\), every morphism \(\bP^1\to X\) is constant, so \(X_K(K)=X(\Fq)\) and \(N^{\ff}(k)=0\) for \(k\ge1\). We therefore assume it.

\begin{prop}\label{prop:ff-geometric}
Suppose \(N^{\ff}(k)=c\,\rho^{k}(1+o(1))\) as \(k\to\infty\), for some \(c>0\) and \(\rho>1\). Then the distribution of \(B-k\) under \(\mu^{\ff}_B\) converges to the geometric distribution \(j\mapsto(1-\rho^{-1})\rho^{-j}\), and
\[
\lim_{B\to\infty}\Ent^{\ff}_B(X) = -\log\bigl(1-\rho^{-1}\bigr)+\frac{\log\rho}{\rho-1}.
\]
\end{prop}

\begin{proof}
The denominator is \(c\rho^{B}(1-\rho^{-1})^{-1}(1+o(1))\), so \(\mu_B^{\ff}(B-j)\to(1-\rho^{-1})\rho^{-j}\) for each fixed \(j\). Moreover \(\mu^{\ff}_B(B-j)\le C\rho^{-j}\) uniformly in \(B\), and \(-x\log x\) is increasing near \(0\), so the terms of the entropy sum are dominated by a summable sequence and the entropy converges to that of the limit. For \(p_j=(1-\tau)\tau^{j}\) with \(\tau=\rho^{-1}\) one computes \(-\sum_jp_j\log p_j=-\log(1-\tau)-\frac{\tau}{1-\tau}\log\tau\).
\end{proof}

\begin{exa}\label{exa:P1}
Let \(\w=(1,1,1)\) and \(X=\{x_0=0\}\iso\bP^1\). A point of \(X_K(K)\) of height \(k\log q\) is a morphism \(\bP^1\to\bP^1\) of degree \(k\) over \(\Fq\), so \(N^{\ff}(0)=q+1\) and \(N^{\ff}(k)=q^{2k+1}-q^{2k-1}\) for \(k\ge1\). Here \(\rho=q^2\), and
\(
\lim_B\Ent^{\ff}_B(X) = -\log(1-q^{-2})+\frac{2\log q}{q^2-1}.
\)
\end{exa}

\begin{rem}\label{rem:phillips}
Two counting problems over \(K\) must be kept apart. In \cite{Phillips2024} the objects of \(\cP_\w(K)\), that is the \(K^\times\)-orbits, are counted by the height with ceilings, which is \(H^{\can}_{\Oo(1)}\) of \cref{sec:6}. The strata \(\cH^{\ff}_k\) above consist instead of coarse points ordered by the stable height; by \cref{prop:ff-spec} this is the Weil height of the image under \(\phi_m\), divided by \(m\).  These are different counting problems.  In the running example the coarse space is \(\bP^1\), and the explicit formula of \cref{exa:P1} applies.  More generally, if several dominant poles occur on the circle \(|u|=\rho^{-1}\), the coefficient asymptotics can acquire a periodic factor, as noted after \cref{prop:entropy-pole}; no single-term asymptotic is asserted here without the hypothesis of \cref{prop:ff-geometric}.
\end{rem}

%------------------------------------------------------------

\subsection{Probabilistic Stratification and Entropy}
We use the probabilistic weighted height \( H_{\w,q}^{\prob} \), based on orbit rarity.
The stratified sets are
\[
\cH_m^{\prob}(X, \F_{q^r}) := \{ \x \in X(\F_{q^r}) \mid H_{\w,q}^{\prob}(\x) = m \},
\]
partitioning \( X(\F_{q^r}) \). By \cref{prop:prob-formula} these are the \textbf{isotropy strata}: \(\x\) and \(\y\) lie in the same stratum if and only if \(g_\x(r)=g_\y(r)\). The height profile function is
\[
\mu_{q^r}^{\prob}(m) := \frac{|\cH_m^{\prob}(X, \F_{q^r})|}{|X(\F_{q^r})|}.
\]
The \textbf{probabilistic height entropy} is
\[
\Ent_{q^r}^{\prob}(X) := -\sum_{m} \mu_{q^r}^{\prob}(m) \log \mu_{q^r}^{\prob}(m),
\]
summing over \( m \) with \(\mu_{q^r}^{\prob}(m) > 0\). The number of strata is at most the number of distinct values of \(g_S(r)\), so \( 0 \leq \Ent_{q^r}^{\prob}(X) \leq \log \#\{g_S(r) : \emptyset\ne S\subseteq T\} \).

\begin{rem}
If \(\gcd(w_i,q^r-1)=1\) for all \(i\), then every \(g_S(r)=1\), there is a single stratum, and \(\Ent_{q^r}^{\prob}(X)=0\). The entropy \(\Ent_{q^r}^{\prob}(X)\) measures isotropy diversity; it is nonzero only in the presence of the Kummer twists of \cite{paper-1}, which are the finite-field counterpart of the gcd-obstructions in \cite{sh-94}.
\end{rem}

\begin{exa}\label{exa:prob}
For the hypersurface of \cref{exa:running} and \(q\) odd there are two strata, the point \([0:0:1]\) and its complement, so
\[
\Ent^{\prob}_{q^r}(X) = \log(q^r+1)-\frac{q^r}{q^r+1}\log q^r ,
\]
which is asymptotic to \(r\log q\cdot q^{-r}\). For \(q\) even the entropy vanishes for all \(r\).
\end{exa}

%-------------------------------------------------------------------------
\subsection{Effect of Reduction of the Weights on Entropy}
The preceding scaling formulas give the corresponding comparison of entropy profiles.

\begin{prop}\label{prop:reduction-entropy}
Let $X \subset \bP_\w^n$ and $X' \subset \bP_{\w'}^n$ be the isomorphic hypersurfaces under the reduction $\w' = \w/d$.
\begin{enumerate}[label=(\roman*)]
\item $\Ent_{q^r}^{\deg}(X) = \Ent_{q^r}^{\deg}(X')$ for all \(r\).
\item $\Ent_{B}^{\ff}(X) = \Ent_{B}^{\ff}(X')$ for all \(B\).
\item $\Ent_{q^r}^{\prob}(X) \le \Ent_{q^r}^{\prob}(X')$, with equality if \(\gcd(d,q^r-1)=1\), and the inequality can be strict.
\end{enumerate}
\end{prop}

\begin{proof}
i) The point sets \(X(\Fqr)\) and \(X'(\Fqr)\) coincide, and by \cref{COR1} the stratum $\cH_m^{\deg}(X,\Fqr)$ is the stratum $\mathcal{H}_{d \cdot m}^{\deg}(X',\Fqr)$. The two profiles differ by the relabeling \(m\mapsto dm\), which does not change the entropy.

ii) By \cref{COR2} the same argument applies; note that \(\lcm(\w')=\lcm(\w)/d\), so the index \(k\) of a stratum is unchanged.

iii) If \(\gcd(d,q^r-1)=1\), then \(\gcd(k_S,q^r-1)=\gcd(k_S/d,q^r-1)\) for every \(S\) and the isotropy strata coincide. In general the relationship between $H_{\w,q}^{\prob}(\x)$ and $H_{\w',q}^{\prob}(\x')$ is point-dependent by \cref{COR3}. We claim that the isotropy stratification for \(\w\) is coarser than that for \(\w'\). Let \(N=q^r-1\) and let \(k_1',k_2'\) be two values of \(k_S/d\) with \(\gcd(k'_1,N)=\gcd(k'_2,N)\). For each prime \(\ell\), either \(v_\ell(k_1')=v_\ell(k'_2)\) or both are at least \(v_\ell(N)\); in either case \(\min(v_\ell(dk'_1),v_\ell(N))=\min(v_\ell(dk'_2),v_\ell(N))\), so \(\gcd(dk_1',N)=\gcd(dk'_2,N)\). Merging strata does not increase the entropy, which gives the inequality. For example, take \(X'\) as in \cref{exa:running} and \(\w=(2,2,4)\), with \(q^r\equiv 3\pmod 4\). For \(\w'\) the point \([0:0:1]\) has \(g=2\) and all others have \(g=1\), so \(\Ent^{\prob}_{q^r}(X')>0\) by \cref{exa:prob}. For \(\w\) the point \([0:0:1]\) has \(g=\gcd(4,q^r-1)=2\) and all others have \(g=\gcd(2,q^r-1)=2\), so there is a single stratum and \(\Ent^{\prob}_{q^r}(X)=0\). When \(q^r\equiv1\pmod 4\) the two entropies agree.
\end{proof}

Thus the sensitivity to a common reduction of the weights is carried by the isotropy, that is, by the stack and not by the coarse variety. This is made precise in \cref{sec:5}.

%*************************************
\subsection{Lang--Weil Asymptotics for Isotropy Entropy}\label{subsec:lang-weil-entropy}
The Lang-Weil estimate \cite{lang-weil} provides an asymptotic for the number of $\Fqr$-rational points on a geometrically irreducible variety $Y$ of dimension $\delta$:
\[
|Y(\Fqr)| = q^{r \delta} + O(q^{r (\delta - 1/2)}),
\]
where the implicit constant depends on the degree and embedding of $Y$.
For the probabilistic height, all but the non-generic isotropy strata have the same value.  The size of the exceptional locus can therefore be estimated by Lang--Weil.

Let \(X\) be geometrically irreducible, let \(S_X\) be the support of its generic point and \(d_X=k_{S_X}\), so that \(\mu_{d_X}\) is the generic stabilizer of \(\mathfrak{X}\); if \(X\) lies in no coordinate hyperplane then \(d_X=d\). Put
\[
\Sigma_r := \{\x\in X \mid g_\x(r)\neq \gcd(d_X,q^r-1)\},\qquad \varepsilon_r := \frac{|\Sigma_r(\Fqr)|}{|X(\Fqr)|}.
\]
Since \(d_X\mid k_S\) for every support \(S\) occurring on \(X\), and \(k_S\mid k_{S'}\) for \(S'\subseteq S\), the set \(\Sigma_r\) is a proper closed subset of \(X\).  Let \(M_{\rm iso}\) be a common period of the finitely many sequences \(g_S(r)\) arising from supports on \(X\).  Then \(\Sigma_r\), and hence its codimension \(c_r\), depends only on \(r\bmod M_{\rm iso}\).

There is a second, independent periodicity.  For each residue class modulo \(M_{\rm iso}\), consider the geometric irreducible components of maximal dimension of the corresponding closed set \(\Sigma_r\), and let \(M_{\rm Frob}\) be the least common multiple of their Frobenius orbit lengths (with \(M_{\rm Frob}=1\) if all the sets are empty).  Put
\[
M:=\operatorname{lcm}(M_{\rm iso},M_{\rm Frob}).
\]
For nonempty \(\Sigma_r\), let \(N_r\) be the number of its maximal-dimensional geometric irreducible components fixed by \(\operatorname{Frob}_q^r\), equivalently the number defined over \(\Fqr\).  Thus both \(c_r\) and \(N_r\) are periodic modulo \(M\).  If \(\Sigma_r=\varnothing\), set \(c_r:=+\infty\), \(N_r:=0\), and \(q^{-r\infty}:=0\).  We also use \(0\log0:=0\), and write \(h(\varepsilon)=-\varepsilon\log\varepsilon-(1-\varepsilon)\log(1-\varepsilon)\).

\begin{thm}\label{thm:prob-asymptotics}
Let \(X \subset \bP_\w^n\) be a geometrically irreducible weighted hypersurface over \(\F_q\) of dimension \(n-1\ge1\), and let \(r\) satisfy \(X(\Fqr)\ne\varnothing\).  Let \(\kappa_r\) be the number of distinct values among the \(g_S(r)\) occurring on \(X(\Fqr)\).  Whenever \(\varepsilon_r\le 1/2\),
\[
h(\varepsilon_r)\le \Ent^{\prob}_{q^r}(X)
\le h(\varepsilon_r)+\varepsilon_r\log(\kappa_r-1)
\]
when \(\kappa_r\ge2\); if \(\kappa_r=1\), then \(\varepsilon_r=0\) and \(\Ent^{\prob}_{q^r}(X)=0\).

Uniformly in \(r\), one has \(\varepsilon_r=O(q^{-r})\) and hence
\[
\Ent^{\prob}_{q^r}(X)=O(rq^{-r}).
\]
More precisely, if \(\Sigma_r\neq\varnothing\) and \(N_r\ge1\), then along each residue class modulo \(M\),
\[
\varepsilon_r=N_rq^{-rc_r}\bigl(1+O(q^{-r/2})\bigr),
\qquad
\Ent^{\prob}_{q^r}(X)
=c_rN_r\,r\log q\,q^{-rc_r}\bigl(1+o(1)\bigr).
\]
The implied constants can be chosen uniformly over the finitely many residue classes modulo \(M\).
\end{thm}

\begin{proof}
If \(\kappa_r=1\), there is a single isotropy stratum and the entropy is zero.  Assume \(\kappa_r\ge2\).  The stratum \(X(\Fqr)\setminus\Sigma_r(\Fqr)\) has mass \(1-\varepsilon_r\).  By the grouping property of Shannon entropy,
\[
\Ent^{\prob}_{q^r}(X)=h(\varepsilon_r)+\varepsilon_r E_r,
\]
where \(E_r\) is the entropy of the conditional distribution on the remaining at most \(\kappa_r-1\) strata.  Hence \(0\le E_r\le\log(\kappa_r-1)\), proving the first assertion.

Only finitely many closed subschemes \(\Sigma_r\) occur as \(r\) varies modulo \(M_{\rm iso}\).  Since each is a proper closed subset of \(X\), Lang--Weil gives the uniform bound \(|\Sigma_r(\Fqr)|=O(q^{r(\dim X-1)})\), while \(|X(\Fqr)|=q^{r\dim X}(1+O(q^{-r/2}))\).  Thus \(\varepsilon_r=O(q^{-r})\), and the entropy bound follows from \(h(\varepsilon)=O(\varepsilon\log(1/\varepsilon))\).

Fix now a residue class modulo \(M\) for which \(\Sigma_r\neq\varnothing\) and \(N_r\ge1\).  On that class the closed set \(\Sigma_r\), its codimension \(c_r\), and the number \(N_r\) of Frobenius-stable top-dimensional geometric components are constant.  Lang--Weil, applied to those components and to their lower-dimensional intersections, gives
\[
|\Sigma_r(\Fqr)|=N_rq^{r(\dim X-c_r)}+O\bigl(q^{r(\dim X-c_r-1/2)}\bigr).
\]
Dividing by the Lang--Weil estimate for \(X\) gives the asserted formula for \(\varepsilon_r\).  Finally,
\(h(\varepsilon)=\varepsilon\log(1/\varepsilon)+O(\varepsilon)\), while the conditional-entropy term is \(O(\varepsilon_r)\); this yields the stated leading term for \(\Ent^{\prob}_{q^r}(X)\).
\end{proof}

\begin{rem}
The period in \cref{thm:prob-asymptotics} has two sources.  The closed set \(\Sigma_r\) is controlled by the periodic isotropy functions \(g_S(r)\), whereas \(N_r\) is controlled by the Frobenius action on the geometric components of that closed set.  These periods need not coincide: \cref{exa:fe-diag} has constant order-two isotropy over odd fields, while its two geometric exceptional points are rational only over extensions in which \(-1\) becomes a square.  The leading entropy term is therefore a Lang--Weil term on arithmetic subsequences, not a function of the isotropy period alone.
\end{rem}

Thus the codimension of the non-generic isotropy locus determines the exponential rate of decay, while the leading coefficient records its Frobenius-stable top-dimensional components.

%****************
\section{Groupoid mass and isotropy}\label{sec:5}
Reduction of the weight vector leaves the coarse variety unchanged but changes the stabilizers.  The quotient-stack description records this distinction and makes the mass identities used below transparent.

\subsection{Weighted Projective Stacks}
Weighted projective spaces admit a natural interpretation as algebraic stacks. The weighted projective space \(\bP_\w^n\) over a field \(k\) is the coarse space of the smooth quotient stack
\[
\cP_\w := [(\bA^{n+1}_k \setminus \{0\}) / \G_m],
\]
where the action of \(\lambda \in \G_m\) is weighted: \(\lambda \cdot (x_0, \dots, x_n) = (\lambda^{w_0} x_0, \dots, \lambda^{w_n} x_n)\). The stabilizers are the group schemes \(\mu_{k_S}\), which are linearly reductive, so \(\cP_\w\) is a tame stack in the sense of \cite{AbramovichOlssonVistoli2008, AbramovichOlssonVistoli2008-2} in every characteristic; it is a Deligne--Mumford stack exactly when \(\ch k\) divides no \(w_i\). The isomorphism classes of objects in the groupoid \(\cP_\w(\Fq)\) map onto the rational points of the coarse space, the fibres being the twists, with each object having automorphisms reflecting the stabilizer under the weighted action; see \cite[Sec.~3.1]{paper-1}.

We use the structure theory of tame stacks from \cite{AbramovichOlssonVistoli2008}.

\begin{rem}[Conventions on the characteristic]\label{rem:char}
We separate three kinds of statements. (a) All statements about the groupoids \(\mathfrak X(\Fqr)\)---twists, rational automorphisms, mass, the heights \(H^{\prob}\) and \(H^{\stack}\), the entropies, the loci \(X^\e\) and \(\Sigma_r\), and the zeta function \(Z^{\can}_H\)---are statements about point counts of the tame Artin stack \(\mathfrak X\) and hold in every characteristic; only the prime-to-\(p\) part of \(k_S\) enters, since \(\gcd(k_S,q^r-1)\) is prime to \(p\). (b) The identification of the loci \(X^\e\) with the components of the inertia stack, the comparison of the cohomology of \(\mathfrak X\) and \(X\), and the word ``twisted sector'' in its usual sense, require \(\mathfrak X\) to be a tame Deligne--Mumford stack, that is \(p\nmid w_i\) for all \(i\); when \(p\mid k_S\) the stabilizer \(\mu_{k_S}\) is non-reduced and the inertia stack is not a disjoint union of such loci. (c) Comparisons with Chen--Ruan cohomology, orbifold Euler characteristics and index theorems refer to the characteristic-zero theory and are made by analogy only. When \(p\) divides some weight we still call \(X^\e\), \(\e\in\Fqr^\times\), a \textbf{rational isotropy locus}, and use ``twisted sector'' only under the hypothesis of (b).
\end{rem}

\subsection{Groupoid Cardinality and Stacky Heights}
For a weighted hypersurface \(\mathfrak{X} \subset \cP_\w\) defined over \(\Fq\), viewed as a closed substack, the \(\Fqr\)-points form a groupoid. For each object \(\xi \in \mathfrak{X}(\Fqr)\), its automorphism group \(\Aut(\xi)(\Fqr)\) is the group of \(\Fqr\)-points of the stabilizer \(\mu_{k_S}\), of order \(g_S(r)\) (\cref{lem:twists}).

We define the \textbf{stacky cardinality}, or mass, of a set \(\mathcal S\) of isomorphism classes in \(\mathfrak{X}(\Fqr)\) as
\[
|\mathcal S|_{\Orb}=\#^{\stack}(\mathcal S) := \sum_{[\xi] \in \mathcal S} \frac{1}{|\Aut(\xi)(\Fqr)|}.
\]
This reduces to ordinary cardinality when all stabilizers are trivial.

\begin{prop}\label{prop:mass}
For every \(r\ge1\) and every \(\x\in X(\Fqr)\), the twists lying over \(\x\) have total mass \(1\). Consequently the pushforward of the mass measure on \(\pi_0(\mathfrak{X}(\Fqr))\) to the coarse space is the counting measure on \(X(\Fqr)\), and
\[
\#^{\stack}(\mathfrak{X}(\F_{q^r})) = |X(\Fqr)| .
\]
The mass of a twist \(\xi\) is \(|\Orb_\w(\tilde\x)|/(q^r-1)\), where \(\tilde\x\) is any vector in the orbit \(\xi\).
\end{prop}

\begin{proof}
By \cref{lem:twists} there are \(g_\x(r)\) twists over \(\x\), each with automorphism group of order \(g_\x(r)\) and orbit of size \((q^r-1)/g_\x(r)\). For \(X=\bP^n_\w\) this is \cite[Prop.~3.9]{paper-1}.
\end{proof}

\begin{exa}\label{exa:mass}
For the hypersurface of \cref{exa:running}, whose stack is the stacky line \(\cP(1,2)\) with a \(\mu_2\)-point, and \(q\) odd: there are \(q\) twists with trivial automorphism group and two twists over \([0:0:1]\), each of mass \(1/2\). Thus the number of twists is \(q+2\), while \(\#^{\stack}(\mathfrak{X}(\Fq)) = q + 1=|X(\Fq)|\).
\end{exa}

Two of the heights of \cref{sec:3} scale by \(d\) under reduction, and each admits a stacky normalization. They are defined on different sets and must be distinguished: \(H^{\deg}_{\w,q}\) is a function on \(X(\overline{\F}_q)\), while \(H^{\ff}_{\w,q}\) is a function on \(X_K(K)\), \(K=\Fq(t)\), and vanishes identically on the constant points \(X(\Fq)\) by \cref{thm:FF-finiteness}. We set
\[
H_{\w,q}^{\stack,\deg}(\x) := |\Aut(\x)| \cdot H^{\deg}_{\w,q}(\x)
\quad (\x\in X(\overline{\F}_q)),
\qquad
H_{\w,q}^{\stack,\ff}(\x) := |\Aut(\x)| \cdot H^{\ff}_{\w,q}(\x)
\quad (\x\in X_K(K)),
\]
where in both cases \(|\Aut(\x)|=k_{S(\x)}\)
is the order of the geometric stabilizer \(\mu_{k_{S(\x)}}\) of \(\x\). \cref{prop:stack-height-inv} below, on invariance under reduction, holds for both, since it is a statement about heights alone. Everything after it---the strata, the profile \(\mu^{\stack}_{q^r}\) and the entropy \(\Ent^{\stack}_{q^r}\)---is formulated for the degree-based height on the finite set \(X(\Fqr)\), and we write \(H^{\stack}_{\w,q}\) for \(H^{\stack,\deg}_{\w,q}\) from here on; \cref{thm:main}(ii) refers to that version. The function-field height does not in general admit an untruncated finite-mass normalization: \(X_K(K)\) may be infinite, and whenever a coarse \(K\)-point has \(k_S>1\) its fibre of twists under \(K^\times/(K^\times)^{k_S}\) is infinite, by \cref{lem:lifting}. Neither is automatic---the quotient is trivial when \(k_S=1\), and \(X_K(K)=X(\Fq)\) when \(X\) is a curve of genus at least \(2\), as noted in \cref{sec:4}---but in general the normalized mass measure need not exist, so we use the bounded-height normalization \(\mu^{\ff}_B\) of \cref{sec:4} instead. On a fixed support stratum, this normalization agrees with reducing the common divisor \(k_S\) from the corresponding subvector of weights.

\begin{prop}\label{prop:stack-height-inv}
Each of \(H_{\w,q}^{\stack,\deg}\) and \(H_{\w,q}^{\stack,\ff}\) is invariant under reduction of the weights. That is, if \(\w' = \w / d\), then \(H_{\w',q}^{\stack,\bullet}(\x') = H_{\w,q}^{\stack,\bullet}(\x)\) for corresponding points \(\x\) and \(\x'\).
\end{prop}

\begin{proof}
Under the reduction \(\w' = \w / d\), the stack \(\cP_{\w}\) is a \(\mu_d\)-gerbe over \(\cP_{\w'}\), and a point of support \(S\) has stabilizer \(\mu_{k_S}\) in the first and \(\mu_{k_S/d}\) in the second \cite[Prop.~4.1]{paper-1}. Thus \(|\Aut(\x')| = |\Aut(\x)|/d\).
The underlying height scales as \(H_{\w',q}(\x')=d\,H_{\w,q}(\x)\) by \cref{COR1} and \cref{COR2}, and the factor \(|\Aut(\x)|\) compensates exactly, yielding \(H_{\w',q}^{\stack}(\x') = H_{\w,q}^{\stack}(\x)\).
\end{proof}

For the probabilistic height the corresponding statement is the following. It identifies \(H^{\prob}_{\w,q}\) as an intrinsic invariant of the groupoid \(\mathfrak{X}(\Fqr)\).

\begin{prop}\label{prop:prob-mass}
Let \(p_r\) be the mass measure on \(\pi_0(\mathfrak{X}(\Fqr))\) normalized to total mass one. For a twist \(\xi\) lying over \(\x\in X(\Fqr)\),
\[
H^{\prob}_{\w,q}(\x) = -\log p_r(\xi) = \log\#^{\stack}(\mathfrak{X}(\Fqr)) + \log|\Aut(\xi)(\Fqr)| .
\]
\end{prop}

\begin{proof}
By \cref{prop:mass}, \(p_r(\xi)=|\Orb_\w(\tilde\x)|/|\hat X(\Fqr)\setminus\{0\}|\), and the second equality is \cref{prop:prob-formula}.
\end{proof}

\subsection{Stacky Entropy}
Using the stacky height, we stratify the groupoid: for each height level \(m\),
\[
\cH_m^{\stack}(\mathfrak{X}, \Fqr) := \left\{ \xi \in \pi_0(\mathfrak{X}(\Fqr)) \;\middle|\; H_{\w,q}^{\stack}(\x) = m \text{ for the coarse point } \x \text{ of } \xi\right\}.
\]
The stacky height profile is
\[
\mu^{\stack}_{q^r}(m) := \frac{\#^{\stack}(\cH_m^{\stack}(\mathfrak{X}, \Fqr))}{\#^{\stack}(\mathfrak{X}(\Fqr))},
\]
a probability measure over the discrete support of possible \(m\). The \textbf{stack-theoretic entropy} is
\[
\Ent^{\stack}_{q^r}(\mathfrak{X}) := - \sum_m \mu^{\stack}_{q^r}(m) \cdot \log \mu^{\stack}_{q^r}(m).
\]
This is the Shannon entropy of the mass profile after the twists have been grouped by their stack-normalized height.

\begin{prop}\label{prop:stack-entropy-inv}
Let \(\mathfrak{X} \subset \cP_\w\) and \(\mathfrak{X}' \subset \cP_{\w'}\) be the stacks of the same hypersurface, with \(\w' = \w/d\). Then:
\[
\Ent^{\stack}_{q^r}(\mathfrak{X}) = \Ent^{\stack}_{q^r}(\mathfrak{X}').
\]
\end{prop}
\begin{proof}
By \cref{prop:mass} the mass of \(\cH_m^{\stack}\) is the number of \(\x\in X(\Fqr)\) with \(H^{\stack}_{\w,q}(\x)=m\), for \(\mathfrak X\) and for \(\mathfrak X'\) alike. The coarse point sets coincide and the stacky heights agree by \cref{prop:stack-height-inv}, so the strata have the same masses,
\[
\#^{\stack}(\cH_m^{\stack}(\mathfrak{X}, \Fqr)) = \#^{\stack}(\cH_m^{\stack}(\mathfrak{X}', \Fqr)),
\]
and similarly for the total mass, yielding identical \(\mu^{\stack}_{q^r}(m)\) and hence equal entropies.
\end{proof}

Note that the stacks \(\mathfrak{X}\) and \(\mathfrak{X}'\) are not isomorphic; the invariance holds because the mass measure forgets the gerbe. The number of twists does not: by contrast, a measure which counts twists with equal weight is sensitive to the reduction, see \cref{prop:gerbe-entropy} below.

\subsection{Stacky Zeta Functions and Entropy Bounds}
Define the \textbf{stacky zeta function} of \(\mathfrak{X}\) by
\[
Z^{\stack}(\mathfrak{X}, t) := \exp\left( \sum_{r=1}^\infty \#^{\stack}(\mathfrak{X}(\F_{q^r})) \frac{t^r}{r} \right).
\]
This is the zeta function adapted to the groupoid cardinality, counting with orbifold weights.

\begin{thm}\label{thm:stack-zeta}
For a weighted hypersurface \(\mathfrak{X} \subset \cP_\w\) over \(\Fq\) with coarse space \(X\),
\[
Z^{\stack}(\mathfrak{X}, t) = Z(X,t).
\]
In particular \(Z^{\stack}(\mathfrak{X}, t)\) is rational and unchanged under reduction of the weights. Each reciprocal zero or pole is a Frobenius eigenvalue on some \(H^i_c(X_{\overline{\F}_q},\Ql)\), hence a \(q\)-Weil integer of some integer weight \(j\) with \(0\le j\le i\le 2(n-1)\): an algebraic integer all of whose complex conjugates have absolute value \(q^{j/2}\). Since \(X\) may be singular the cohomology is mixed, and \(j<i\) can occur.
\end{thm}

\begin{proof}
The equality is \cref{prop:mass}. The coarse space \(X\) is a projective variety over \(\Fq\), so \(Z(X,t)\) is rational by Dwork's theorem \cite{dwork1960}, by the Grothendieck--Lefschetz trace formula its reciprocal zeros and poles are Frobenius eigenvalues on the groups \(H^i_c(X_{\overline{\F}_q},\Ql)\), and by Deligne's theorem \cite{deligne-weil2} each eigenvalue on \(H^i_c\) is an algebraic integer which is pure of some integer weight \(j\le i\). The coarse varieties of \(\w\) and \(\w/d\) coincide.
\end{proof}

\begin{rem}
When \(p\nmid w_i\) for all \(i\), the mass identity can also be read cohomologically from the trace formula for tame Deligne--Mumford stacks \cite{Behrend}.  The twist zeta function is different: it counts isomorphism classes without inverse-automorphism weights.  Its functional equation is treated in \cref{subsec:fe}.
\end{rem}

\begin{prop}
The stacky entropy satisfies \(0 \leq \Ent^{\stack}_{q^r}(\mathfrak{X}) \leq \log |\Supp(\mu^{\stack}_{q^r})|\), with equality on the left for a single stratum and on the right for the uniform profile.
\end{prop}

\begin{proof}
This is the standard entropy bound for a probability distribution on a finite set. Equality on the right occurs precisely for the uniform distribution on its support.
\end{proof}

When \(p\nmid w_i\), the fixed loci that determine the isotropy counts are the coarse spaces of the components of the inertia stack. For \(\e\in\overline{\F}_q^\times\) put
\[
X^{\e} := \{\x\in X \mid \e^{w_i}=1 \text{ for all } i\in S(\x)\},
\]
a closed subvariety of \(X\) cut out by coordinate hyperplanes, which is defined over \(\Fqr\) when \(\e\in\Fqr^\times\). When \(p\nmid w_i\) for all \(i\), the inertia stack of \(\mathfrak X\) is the disjoint union of the substacks over the \(X^\e\), which are then its twisted sectors. The following proposition is a statement about point counts and needs no such hypothesis.

\begin{prop}\label{prop:inertia}
The isotropy stratification of \(X(\Fqr)\), the profile \(\mu^{\prob}_{q^r}\), and the entropy \(\Ent^{\prob}_{q^r}(X)\) are determined by the numbers \(|X^{\e}(\Fqr)|\), \(\e\in\Fqr^\times\), together with the inclusions among the \(X^\e\). Moreover
\[
\#\pi_0(\mathfrak{X}(\Fqr)) = \sum_{\e\in\Fqr^\times}|X^{\e}(\Fqr)|,
\qquad
\Sigma_r = \bigcup_{\e\in\Fqr^\times,\ \e^{d_X}\neq1} X^{\e},
\]
where \(\Sigma_r\) is the locus of \cref{thm:prob-asymptotics}.
\end{prop}

\begin{proof}
For \(\x\in X(\Fqr)\) the set \(\{\e\in\Fqr^\times : \x\in X^\e\}\) is \(\mu_{k_{S(\x)}}(\Fqr)\), of order \(g_\x(r)\). Hence \(g_\x(r)\) is recovered from the \(X^\e\) containing \(\x\), and the number of points with a given isotropy is obtained from the \(|X^\e(\Fqr)|\) by inclusion--exclusion. Summing \(g_\x(r)\) over \(\x\) gives the first formula by \cref{lem:twists}; this is the Burnside form of the orbit count \cite[Lem.~3.1]{paper-1}. Finally \(g_\x(r)\neq\gcd(d_X,q^r-1)\) if and only if \(\mu_{k_{S(\x)}}(\Fqr)\) contains an element outside \(\mu_{d_X}\).
\end{proof}

Thus the entropy of the isotropy stratification is an invariant of the rational isotropy loci---of the rational twisted sectors of the inertia stack when \(p\nmid w_i\)---and by \cref{thm:prob-asymptotics} its leading asymptotic is governed by the top-dimensional components of the non-generic ones.

\begin{exa}
For \(\cP(1,2)\) and odd \(q\), the inertia stack is \(\cP(1,2)\sqcup B\mu_2\), with the second component over the stacky point and corresponding to \(\e=-1\). Accordingly
\[
|X^{1}(\Fq)|+|X^{-1}(\Fq)|=(q+1)+1=q+2,
\]
which is the number of twists, and \(\Ent^{\prob}_{q}=h(1/(q+1))\) as in \cref{exa:prob}.
\end{exa}

The Kummer groups appearing here are the finite-field counterparts of the root-of-unity torsors considered in \cite{sh-94}.  Over \(\F_{q^r}\) their cardinalities are the integers \(g_S(r)\), and these cardinalities enter the twist count directly.

Two entropies are now attached to the mass measure, and they should not be confused. The stacky entropy \(\Ent^{\stack}_{q^r}\) above is the entropy of the mass measure \emph{after grouping the twists into height strata}; it is small and invariant under reduction. The following is the entropy of the mass measure on the set of twists itself, with no grouping; it is of size \(\log|X(\Fqr)|\) and records the gerbe.

\begin{defn}
The \textbf{groupoid entropy}, or canonical entropy, of the groupoid \(\mathfrak{X}(\Fqr)\) is the Shannon entropy of the normalized mass measure \(p_r\) on the set of twists,
\[
\Ent^{\can}_{q^r}(\mathfrak{X}) := -\sum_{\xi\in\pi_0(\mathfrak{X}(\Fqr))} p_r(\xi)\log p_r(\xi).
\]
\end{defn}

\begin{prop}\label{prop:gerbe-entropy}
Let \(\mathfrak{X}\) be the stack of \(X \subset \bP_\w^n\). Then
\[
\Ent^{\can}_{q^r}(\mathfrak{X}) = \log |X(\Fqr)| + \frac{1}{|X(\Fqr)|}\sum_{\x\in X(\Fqr)}\log g_\x(r),
\]
the mean of the probabilistic height. If \(\w'=\w/d\) and \(\mathfrak{X}'\subset\cP_{\w'}\), then
\[
0\;\le\;\Ent^{\can}_{q^r}(\mathfrak{X})-\Ent^{\can}_{q^r}(\mathfrak{X}')\;\le\;\log \gcd(d, q^r-1),
\]
with equality on the right when \(\gcd(d,k_S/d)=1\) for every support \(S\) occurring on \(X(\Fqr)\), where the log term captures the entropy of the \(\mu_d\)-gerbe.
\end{prop}

\begin{proof}
Over \(\x\) there are \(g_\x(r)\) twists, each with \(p_r(\xi)=1/(g_\x(r)|X(\Fqr)|)\) by \cref{prop:prob-mass}; their contribution to the entropy is \(|X(\Fqr)|^{-1}\bigl(\log|X(\Fqr)|+\log g_\x(r)\bigr)\). The comparison is the average of \cref{COR3} over \(X(\Fqr)\).
\end{proof}

\begin{rem}
Let \(\nu_r\) be the uniform probability measure on the set of twists and \(A_X(q^r)=\#\pi_0(\mathfrak{X}(\Fqr))\). The same computation gives the relative entropy
\[
D(\nu_r\,\|\,p_r) = \mathbb{E}_{\nu_r}\bigl[\log|\Aut(\xi)(\Fqr)|\bigr]-\log\frac{A_X(q^r)}{|X(\Fqr)|},
\]
which vanishes exactly when all points of \(X(\Fqr)\) have the same number of twists. By \cref{thm:prob-asymptotics} the mean in \cref{prop:gerbe-entropy} is \(\log\gcd(d_X,q^r-1)+O(q^{-rc_r})\), so \(\Ent^{\can}_{q^r}(\mathfrak X)-\log|X(\Fqr)|\) is asymptotically periodic in \(r\).
\end{rem}

The formulas above will be used only through their consequences for the isotropy-weighted point counts in \cref{sec:9}.

%***********************************************
\section{Comparison with stack heights}\label{sec:6}
We compare the function-field height of \cref{sec:3} with the height on algebraic stacks defined by Ellenberg, Satriano, and Zureick-Brown \cite{ellenberg}.  The comparison is made on the ambient weighted projective stack; for a hypersurface we restrict the ambient height along the closed immersion.

Let \(K = \F_q(t)\) be the function field of \(C=\bP^1\) over \(\F_q\), a global field. Consider the weighted projective stack \(\cP_{\w,K} = [(\bA^{n+1}_K \setminus \{0\}) / \G_{m,K}]\), and the hypersurface stack \(\mathfrak{X}_K\), the substack defined by the zero locus of \(f\). Both are defined over \(\Fq\), hence extend to the constant families \(\cP_\w\times C\) and \(\mathfrak X\times C\) over \(C\). The ambient family \(\cP_\w\times C\) is normal, proper over \(C\) and has finite diagonal, as required in \cite[Sec.~2]{ellenberg}. An arbitrary weighted hypersurface need not be normal, so we do not apply that theory to \(\mathfrak X\) directly: we define the heights of \(\mathfrak X_K\) by restriction along the closed immersion \(\iota\colon\mathfrak X\hookrightarrow\cP_\w\), setting
\[
H^{\can}_{\Oo(k)}(x) := H^{\can}_{\Oo_{\cP_\w}(k)}(\iota\circ x),
\qquad
H^{\can,st}_{\Oo(k)}(x) := H^{\can,st}_{\Oo_{\cP_\w}(k)}(\iota\circ x),
\]
which is the value given by functoriality \cite[Rem.~2.16]{ellenberg} whenever \(\mathfrak X\) is itself normal. All statements below are for these heights.

By \cref{lem:lifting}(i), an object \(x\) of \(\mathfrak{X}_K(K)\) is a \(K^\times\)-orbit on \(\hat X(K)\setminus\{0\}\); we write \(\tilde\x\) for a vector in this orbit and \(\x\in X_K(K)\) for its coarse point. By \cref{lem:lifting}(iii) the map \(x\mapsto\x\) is surjective with fibres the torsors under \(K^\times/(K^\times)^{k_S}\); the fibre may be infinite, and is infinite precisely when \(k_S>1\). The height of \cite{ellenberg} is a function of the object \(x\); the function-field height of \cref{sec:3} is a function of the coarse point \(\x\). Let \(\cV\) be a vector bundle on \(\mathfrak{X}\), such as a power of the tautological bundle \(\Oo(1)\).

\begin{defn}\label{def:can-height}
For \(x \in \mathfrak{X}_K(K)\) and a vector bundle \(\cV\), the \textbf{height} and the \textbf{stable height} of \cite[Def.~2.11, Def.~2.12]{ellenberg} are
\[
H^{\can}_{\cV}(x) = -\deg(\pi_* \bar x^* \cV^\vee),
\qquad
H^{\can,st}_{\cV}(x) = -\deg_{\cC}(\bar x^* \cV^\vee),
\]
where \((\cC, \bar x, \pi)\) is a tuning stack for \(x\): \(\cC\) is a normal Artin stack with finite diagonal, \(\pi: \cC \to C\) is a birational coarse space map, and \(\bar x: \cC \to \mathfrak{X}\times C\) extends the given \(x: \spec K \to \mathfrak{X}_K\).
\end{defn}

\noindent\textbf{Normalizations.} We follow \cite[Sec.~2.2]{ellenberg}: the degree of a divisor \(\sum n_PP\) on \(C\) is \(\sum n_P\log|\kappa(P)|\), so degrees on \(C\) lie in \((\log q)\Z\), degrees on \(\cC\) lie in \((\log q)\Q\), and \(q_v=q^{\deg v}\), \(|a|_v=q_v^{-\ord_v(a)}\) as in \cref{sec:3}. With these conventions no factor of \(\log q\) intervenes in the comparisons below.

Both heights are independent of the choice of tuning stack \cite[Prop.~2.13]{ellenberg}. They differ by a sum of local discrepancies supported at the stacky points of \(\cC\), \(H^{\can}_\cV=H^{\can,st}_{\cV}+\sum_v\delta_{\cV;v}\) \cite[(2.22)]{ellenberg}; the pushforward \(\pi_*\) is what makes the height non-additive in \(\cV\), whereas the stable height of a line bundle is additive, \(H^{\can,st}_{\cL^{\otimes k}}=k\,H^{\can,st}_{\cL}\), because pullback to \(\cC\) commutes with tensor powers.

For a vector \(\tilde\x\in K^{n+1}\setminus\{0\}\) and a place \(v\) put
\[
a_v(\tilde\x) := \max_{0\le i\le n}\frac{\log_{q_v}|\tilde x_i|_v}{w_i} = -\min_{i\in S}\frac{\ord_v(\tilde x_i)}{w_i}\ \in\ \tfrac1m\Z,\qquad m=\lcm(\w),
\]
which vanishes for all but finitely many \(v\), so that \(H^{\ff}_{\w,q}(\x)=\sum_va_v(\tilde\x)\log q_v\).

\begin{prop}[Specialization to Function-Field Height]\label{prop:ff-spec}
Let \(x\in\mathfrak{X}_K(K)\) be the orbit of \(\tilde\x\), with coarse point \(\x\). Then:
\begin{enumerate}[label=(\roman*)]
\item \(H^{\can,st}_{\Oo(1)}(x) = H_{\w,q}^{\ff}(\x)\), and \(H^{\can}_{\Oo(m)}(x)=m\,H_{\w,q}^{\ff}(\x)\);
\item \(H^{\can}_{\Oo(1)}(x) = \sum_{v}\bigl\lceil a_v(\tilde\x)\bigr\rceil\log q_v\);
\item the local discrepancy at \(v\) is \(\delta_{\Oo(1);v}(x)=\bigl(\lceil a_v(\tilde\x)\rceil-a_v(\tilde\x)\bigr)\log q_v\), which lies in \([0,\log q_v)\) and, when nonzero, is at least \(\frac1m\log q_v\).
\end{enumerate}
The stable height depends only on the coarse point \(\x\); the height \(H^{\can}_{\Oo(1)}\) depends on the twist \(x\).
\end{prop}

\begin{proof}
Write \(\cL=\Oo(1)\). The coordinate \(x_i\) is a section of \(\cL^{w_i}\), so the \(n+1\) sections \(s_i=x_i^{m/w_i}\) of \(\cL^{m}\) have no common zero on \(\bA^{n+1}\setminus\{0\}\); equivalently the map \(\Oo_{\cP_\w}^{\oplus(n+1)}\to\cL^m\) they define is surjective. Hence \(\cL^m\) is globally generated by \(s_0,\dots,s_n\) on the whole model \(\cP_\w\times C\), and not merely generically globally generated in the sense of \cite[Sec.~2.4]{ellenberg}; this is the situation of \cite[Sec.~3.3]{ellenberg}. Since \(K\) is a function field, \(C\) has no Archimedean places. Both hypotheses of the second, exact assertion of \cite[Prop.~2.28]{ellenberg} therefore hold, so that proposition applies with no bounded error term \(E(x)\). Choosing the identification of \(x^*\cL\) with \(K\) given by \(\tilde\x\), the pullbacks of the \(s_i\) are \(\tilde x_i^{m/w_i}\), and \cite[Prop.~2.28]{ellenberg}, applied to the power \(\cL^m\), gives
\[
H^{\can}_{\cL}(x)=\sum_v\Bigl\lceil\frac1m\log_{q_v}\max_i\bigl|\tilde x_i^{\,m/w_i}\bigr|_v\Bigr\rceil\log q_v=\sum_v\lceil a_v(\tilde\x)\rceil\log q_v,
\]
which is (ii); this is \cite[(3.5)]{ellenberg} for the weights \(\w\).

The Veronese morphism \(\phi_m\colon\cP_\w\to\bP^n\) satisfies \(\phi_m^*\Oo_{\bP^n}(1)=\cL^m\).  Since a bundle pulled back from a scheme has equal height and stable height \cite[Prop.~2.15]{ellenberg}, and stable height is additive in tensor powers,
\[
m H^{\can,st}_{\cL}(x)
 =H^{\can,st}_{\cL^m}(x)
 =H^{\can}_{\cL^m}(x)
 =h(\phi_m(\x)).
\]
By \cref{sec:3}, \(h(\phi_m(\x))=mH^{\ff}_{\w,q}(\x)\), proving the first identity in (i); the same display gives the second.  Subtracting this stable height from the ceiling formula in (ii) gives
\[
\delta_{\cL;v}(x)=\bigl(\lceil a_v(\tilde\x)\rceil-a_v(\tilde\x)\bigr)\log q_v.
\]
Since \(a_v(\tilde\x)\in\frac1m\Z\), every nonzero discrepancy is at least \(\frac1m\log q_v\).  Replacing \(\tilde\x\) by the twist \((a^{w_i/k_S}\tilde x_i)\) changes \(a_v\) by \(-\ord_v(a)/k_S\); the sum of the stable local terms is unchanged by the product formula, while the ceiling terms can change.
\end{proof}

Under the gerbe \(\pi\colon\cP_\w\to\cP_{\w'}\), \(\w'=\w/d\), one has \(\pi^*\Oo(1)\simeq\Oo(d)\), and functoriality \cite[Rem.~2.16]{ellenberg} together with additivity of the stable height gives \(H^{\can,st}_{\Oo_{\w'}(1)}(\pi\circ x)=d\,H^{\can,st}_{\Oo_{\w}(1)}(x)\), which is \cref{COR2}.

The Northcott property for \(H^{\can}_{\Oo(1)}\) does not follow from that of \(H^{\ff}\) alone, since infinitely many twists lie over each coarse point. It rests on the following lemma.

\begin{lem}\label{lem:kummer-finite}
Let \(L\) be a global function field with finite constant field \(\F\), let \(k\ge1\) and \(B\ge0\). The set of classes \(a\in L^\times/(L^\times)^k\) such that
\[
\sum_{v\,:\,k\,\nmid\,\ord_v(a)}\deg v\;\le\;B
\]
is finite.
\end{lem}

\begin{proof}
Consider the homomorphism \(\partial\colon L^\times/(L^\times)^k\to\Div(L)/k\Div(L)\), \(a\mapsto\operatorname{div}(a)\bmod k\). A class in \(\Div(L)/k\Div(L)\) is a function from the places to \(\Z/k\) with finite support, and the condition says that the support of \(\partial(a)\) has degree at most \(B\). There are finitely many places of degree at most \(B\), hence finitely many such classes. It remains to see that \(\ker\partial\) is finite. If \(\operatorname{div}(a)=kD\), then \(\deg D=0\) and the class of \(D\) lies in \(\Pic^0(L)[k]\), which is finite because \(\Pic^0(L)\) is finite for a finite constant field. If this class is trivial, \(D=\operatorname{div}(b)\) and \(a=ub^k\) with \(u\in\F^\times\), so the class of \(a\) lies in the image of \(\F^\times/(\F^\times)^k\), which is finite. Thus \(\ker\partial\) is an extension of a subgroup of \(\Pic^0(L)[k]\) by a quotient of \(\F^\times/(\F^\times)^k\).
\end{proof}

\begin{thm}[Properties of the Canonical Height]\label{thm:can-properties}
On \(\mathfrak{X}_K\) the heights of \cref{def:can-height} satisfy:
\begin{enumerate}[label=(\roman*)]
\item \textbf{Finiteness (Northcott Property)}: For \(\cV=\Oo(1)\), the set \[ \{ x \in \mathfrak{X}_K(L) \mid H^{\can}_{\cV}(x) \leq B \}\] is finite for any finite extension \(L/K\) and \(B > 0\).
\item \textbf{Base change}: Under a separable extension \(L/K\), \(H^{\can,st}_{\cV}(x_L) = [L:K] \cdot H^{\can,st}_{\cV}(x)\). The unstable height \(H^{\can}_\cV\) does not satisfy such a relation in general.
\item \textbf{Functoriality}: If \(f: \mathfrak{X} \to \mathfrak{Y}\) is a morphism of stacks and \(\cW\) a vector bundle on \(\mathfrak{Y}\), then \(H^{\can}_{f^* \cW}(x) = H^{\can}_{\cW}(f \circ x)\).
\item \textbf{Schemes}: If \(\cV\) is pulled back from a scheme, then \(H^{\can}_\cV=H^{\can,st}_\cV\).
\end{enumerate}
\end{thm}

\begin{proof}
(ii) is \cite[Prop.~2.14]{ellenberg}, (iii) is \cite[Rem.~2.16]{ellenberg}, and (iv) is \cite[Prop.~2.15]{ellenberg}.

For (i) we work over \(L\), with the places of \(L\), and we let \(q\) denote the order of its constant field; \cref{prop:ff-spec} holds over \(L\) with the same proof. Since \(H^{\can}_{\Oo(1)}(x)\ge H^{\ff}_{\w,q}(\x)\), \cref{thm:FF-finiteness} over \(L\) shows that only finitely many coarse points \(\x\) occur. Fix one, of support \(S\), with representative \(\tilde\x\), put \(k=k_S\), and let \(V_0\) be the finite set of places with \(a_v(\tilde\x)\notin\Z\). By \cref{lem:lifting}(iii) the objects over \(\x\) are the orbits \(x_a\) of the vectors \((a^{w_i/k}\tilde x_i)\), \(a\in L^\times/(L^\times)^k\), and
\[
H^{\can}_{\Oo(1)}(x_a)-H^{\ff}_{\w,q}(\x)=\sum_v\delta_v(a),\qquad
\delta_v(a)=\bigl(\lceil y_v\rceil-y_v\bigr)\log q_v,\quad y_v=a_v(\tilde\x)-\frac{\ord_v(a)}{k} .
\]
Since \(k\mid w_i\mid m\) for \(i\in S\), the number \(y_v\) lies in \(\frac1m\Z\), so \(\delta_v(a)\ge\frac1m\log q_v\) whenever \(y_v\) is not an integer. For \(v\notin V_0\) this happens exactly when \(k\nmid\ord_v(a)\). Hence \(H^{\can}_{\Oo(1)}(x_a)\le B\) implies
\[
\sum_{v\,:\,k\,\nmid\,\ord_v(a)}\deg v\;\le\;\frac{mB}{\log q}+\sum_{v\in V_0}\deg v,
\]
and \cref{lem:kummer-finite} leaves finitely many classes \(a\).
\end{proof}

Thus \(H^{\ff}\) is the stable part of the stack height, while \(H^{\prob}\) is the information content \(-\log p_r(\xi)\) of a twist in the special fibre.  The degree-based function has a different origin.

\begin{rem}[The degree-based height]\label{rem:deg-spec}
The degree-based height \(H_{\w,q}^{\deg}\) is not a specialization of \(H^{\can}_\cV\); see \cref{rem:deg-status}. A point of \(X(\Fqr)\), viewed as a constant point over \(\Fqr(t)\), has function-field height zero by \cref{thm:FF-finiteness}. The two functions therefore measure different quantities: \(H^{\ff}\) is a sum of local weighted maxima over the places of \(K\), whereas \(H^{\deg}\) records coordinate extension degrees, with the additional factor \(r_\x\).
\end{rem}

%*********************
\subsection{Entropy and rational isotropy}\label{subsec:stacky-entropy}

The entropies attached to the groupoid \(\mathfrak{X}(\Fqr)\) depend only on the rational isotropy loci, that is, when \(p\nmid w_i\), on the inertia stack (\cref{rem:char}). Recall that the inertia stack is the stack of pairs \((x, g)\) where \(x \in \mathfrak{X}\) and \(g \in \Aut(x)\) is an automorphism (stabilizer element). For weighted hypersurfaces, stabilizers arise from the \(\G_{m}\)-action, and the fixed loci \(\Fix(\e)=X^\e\) record the corresponding rational isotropy. By \cref{prop:inertia}, the profile \(\mu^{\prob}_{q^r}\) and the entropy \(\Ent^{\prob}_{q^r}(X)\) are functions of the point counts \(|\Fix(\e)(\Fqr)|\), \(\e\in\Fqr^\times\), and by \cref{prop:gerbe-entropy} the same holds for \(\Ent^{\can}_{q^r}(\mathfrak X)\), since \(g_\x(r)=\#\{\e\in\Fqr^\times:\x\in\Fix(\e)\}\).

Under reduction of the weights the three entropies behave differently. The stack entropy \(\Ent^{\stack}_{q^r}(\mathfrak{X})\) is invariant (\cref{prop:stack-entropy-inv}); the groupoid entropy \(\Ent^{\can}_{q^r}(\mathfrak{X})\) decreases by at most \(\log\gcd(d,q^r-1)\) (\cref{prop:gerbe-entropy}); and \(\Ent^{\prob}_{q^r}(X)\) does not decrease (\cref{prop:reduction-entropy}). These statements quantify separately the change in the generic stabilizer and the refinement of the isotropy strata.

%*****

\subsection{The function-field height zeta series}\label{subsec:height-zeta-entropy}

For \(X_K(K)\) we use the Laplace transform of the stable height. Since \(H^{\ff}\) takes values in \(\frac{\log q}{m}\Z_{\ge0}\), the natural variable is \(u=q^{-s/m}\).

\begin{defn}
The \textbf{function-field height zeta function} of \(X\) over \(K=\Fq(t)\) is the power series, indexed by the coarse points,
$$Z_{H^{\ff}}(X_K, s) = \sum_{\x \in X_K(K)} e^{-s H^{\ff}_{\w,q}(\x)} = \sum_{k\ge0} N^{\ff}(k)\,u^{k},\qquad u=q^{-s/m},$$
with \(N^{\ff}(k)\) as in \cref{sec:4}. The stacky version \(Z_{H^{\can}}(\mathfrak{X}_K, s)\) is the sum of \(e^{-sH^{\can}_{\Oo(1)}(x)}/|\Aut(x)(K)|\) over isomorphism classes \(x\in\mathfrak X_K(K)\); each level set is finite by \cref{thm:can-properties}(i). We do not study it here.
\end{defn}

Its constant term is \(N^{\ff}(0)=|X(\Fq)|\). Height zeta functions over function fields occur in related counting problems; see, for example, \cite{loeser-motzeta,chambert-loir-tschi}.

\begin{prop}\label{prop:entropy-pole}
Assume that $Z_{H^{\ff}}(X_K, s)$, as a function of \(u\), is meromorphic in a disc \(|u|<\rho^{-1}+\eta\) with a single pole there, simple and located at \(u=\rho^{-1}\), \(\rho>1\). Then \(N^{\ff}(k)=c\rho^k(1+o(1))\), the abscissa of convergence is \(s_0=m\log\rho/\log q\), and
\[
\lim_{B\to\infty}\Ent_B^{\ff}(X) = -\log\bigl(1-\rho^{-1}\bigr)+\frac{\log\rho}{\rho-1}.
\]
\end{prop}

\begin{proof}
Subtracting the principal part \(c/(1-\rho u)\) leaves a function holomorphic in the larger disc, whose coefficients are \(O((\rho^{-1}+\eta')^{-k})\) for some \(\eta'>0\). The rest is \cref{prop:ff-geometric}.
\end{proof}

\begin{rem}
Thus, under the stated single-pole hypothesis, the limiting entropy is determined by the dominant pole. In \cref{exa:P1}, \(Z_{H^{\ff}}=(q+1)+(q-q^{-1})\,q^2u/(1-q^2u)\) is rational with its pole at \(u=q^{-2}\). The coefficients are the cardinalities of the strata $\cH_k^{\ff}$. If several poles lie on the circle \(|u|=\rho^{-1}\), the coefficient asymptotics acquire a periodic factor; the normalized profile may then fail to converge and the entropy may be asymptotically periodic in \(B\) rather than convergent.
\end{rem}

%****************************************************************
\section{Power sums of the mass measure}\label{sec:7}

On the finite set of twists over \(\F_{q^r}\), the probabilistic height is \(-\log p_r\), where \(p_r\) is the normalized mass measure.  We record the standard power-sum identities needed later; the arithmetic input is the formula for \(p_r\) in terms of the isotropy orders.

\subsection{Power sums and escort measures}

\begin{defn}
The \textbf{mass power sum} over \(\Fqr\) is
\[
Z^{\can}(\beta;\Fqr)=\sum_{\xi\in\pi_0(\mathfrak X(\Fqr))}p_r(\xi)^\beta,
\]
where \(\beta\ge0\). For the degree-based height we similarly write
\[
Z^{\deg}(\beta;\Fqr)=\sum_m e^{-\beta m}|\cH_m^{\deg}(X,\Fqr)|.
\]
\end{defn}

By \cref{prop:prob-mass} and \cref{lem:twists},
\begin{equation}\label{eq:partition}
Z^{\can}(\beta;\Fqr) = |X(\Fqr)|^{-\beta}\sum_{\x\in X(\Fqr)} g_\x(r)^{1-\beta}.
\end{equation}
At \(\beta=0\) this is the number of twists \(A_X(q^r)\), and at \(\beta=1\) it equals \(1\).  Put
\[
p_{r,\beta}(\xi)=\frac{p_r(\xi)^\beta}{Z^{\can}(\beta;\Fqr)}.
\]
This is the usual escort distribution.  Its entropy is
\[
S(\beta)=-\sum_\xi p_{r,\beta}(\xi)\log p_{r,\beta}(\xi)
= -\beta\,\frac{\partial}{\partial\beta}\log Z^{\can}(\beta;\Fqr)
 +\log Z^{\can}(\beta;\Fqr).
\]

\subsection{Effect of weight reduction}

The reduction \(\w \to \w / d\) rigidifies the subgroup \(\mu_d\) of the generic stabilizer, which is \(\mu_{d_X}\) with \(d\mid d_X\), equality holding when \(X\) lies in no coordinate hyperplane. By \cref{prop:mass} and \cref{COR3} the value \(Z^{\can}(1;\Fqr)=1\) and the pushforward of \(p_r\) to the coarse space are unchanged, while \(Z^{\can}(\beta;\Fqr)\) for \(\beta\ne1\) records the gerbe.

The entropy identity is the following.

\begin{thm}[Entropy identity]\label{thm:duality}
For the mass power sum \(Z^{\can}(\beta;\Fqr)\),
\[
\Ent_{q^r}^{\can}(\mathfrak{X}) = S(1) = -\left. \frac{\partial}{\partial \beta} \log  Z^{\can}(\beta; \F_{q^r}) \right|_{\beta=1},
\qquad
S(0):=\lim_{\beta\downarrow0}S(\beta)=\log A_X(q^r),
\]
Thus the groupoid entropy is the mean probabilistic height.  At \(\beta=0\) the escort distribution is uniform on the set of twists, while at \(\beta=1\) it is the mass measure.
\end{thm}

\begin{proof}
Since \(Z^{\can}(1;\Fqr)=1\) and \(p_{r,1}=p_r\), the displayed entropy formula gives \(S(1)=\Ent^{\can}_{q^r}(\mathfrak X)\).  Differentiating \(\log\sum_\xi p_r(\xi)^\beta\) at \(\beta=1\) gives the same identity.  At \(\beta=0\), \(p_{r,0}\) is uniform on the \(A_X(q^r)\) twists. 
\end{proof}

\begin{exa}
For the hypersurface of \cref{exa:running} with \(q\) odd, \eqref{eq:partition} gives \(Z^{\can}(\beta; \F_q) = (q+1)^{-\beta}\bigl(q+2^{1-\beta}\bigr)\), and
\[
\Ent^{\can}_q(\mathfrak X)=\log(q+1)+\frac{\log 2}{q+1},\qquad S(0)=\log(q+2).
\]
For the weights \((2,2,4)\) and \(q\equiv1\pmod4\) one finds \(\Ent^{\can}_q=\log(q+1)+\frac{q\log2+\log4}{q+1}\), which exceeds the previous value by exactly \(\log 2=\log\gcd(d,q-1)\).
\end{exa}

\subsection{Asymptotic behavior in the extension degree}

By \eqref{eq:partition} and \cref{thm:prob-asymptotics},
\[
\log Z^{\can}(\beta;\Fqr)=(1-\beta)\log M_r+O(\varepsilon_r),
\qquad
M_r:=|X(\Fqr)|\cdot\gcd(d_X,q^r-1),
\]
uniformly for \(\beta\) in compact subsets of \(\R_{\ge0}\).  On a residue class for which \(\Sigma_r\neq\varnothing\), this error is \(O(q^{-rc_r})\); if \(\Sigma_r=\varnothing\), the formula is exact. For each fixed \(r\), the finite sum extends to an entire function of the complex variable \(\beta\).  The estimate shows that the contribution of the non-generic isotropy locus is of order \(q^{-rc_r}\), whereas \(\Ent^{\prob}_{q^r}(X)\) has the additional factor of order \(r\).  The identities above are standard properties of power sums of a finite probability distribution; the arithmetic content is the formula \eqref{eq:partition}.

The generating function in the extension degree \(r\) is the two-variable series of \cref{sec:9}.

%***********************
\section{Height transforms}\label{sec:8}
We use three elementary transforms of the height functions introduced above.  They serve different purposes and will be kept notationally separate.

\subsection{Three transforms}
Three different transforms of a height appear in this paper, and they should be kept apart.

\begin{defn}\label{def:taxonomy}
Let \(\mathcal S\) be a finite or countable set with a function \(H\colon\mathcal S\to\R_{\ge0}\) having finite level sets, and let \(\mu\) be a probability measure with finite support.
\begin{enumerate}[label=(\alph*)]
\item If \(H>0\) on \(\mathcal S\), the \textbf{Dirichlet-type height series} is \(D_H(s)=\sum_{\x\in\mathcal S}H(\x)^{-s}\).
\item The \textbf{Laplace-type height series} is \[L_H(s)=\sum_{\x\in\mathcal S}e^{-sH(\x)}.\]
\item The \textbf{profile zeta function} of \(\mu\) is \(\zeta_\mu(s)=\sum_m\mu(m)^s\).
\end{enumerate}
\end{defn}

The three are used as follows. The Dirichlet type (a) is defined for the degree-based and the probabilistic heights on \(\mathcal S=X(\Fqr)\). Here \(H^{\deg}_{\w,q}>0\) always, by \cref{prop:deg-bounds}. By \cref{prop:prob-formula}, \(H^{\prob}_{\w,q}(\x)=\log|X(\Fqr)|+\log g_\x(r)\), which vanishes precisely when \(|X(\Fqr)|=1\) and \(g_\x(r)=1\); if the unique point carries several twists the height is again positive. We therefore assume \(X(\Fqr)\neq\varnothing\), and exclude the single degenerate case just described, in which the Dirichlet transform must be replaced by the Laplace one. The degenerate case does not occur for \(r\) large when \(X\) is geometrically irreducible of positive dimension, since then \(|X(\Fqr)|\to\infty\) by Lang--Weil. Under these assumptions:
\[
Z_{H_{\w,q}}(X, s; q^r) := \sum_{\x \in X(\Fqr)} \frac{1}{H_{\w,q}(\x)^s},\qquad H_{\w,q}\in\{H^{\deg}_{\w,q},\,H^{\prob}_{\w,q}\}.
\]
It is \emph{not} defined for the function-field height, which lives on \(X_K(K)\) and vanishes exactly on \(X(\Fq)\) (\cref{thm:FF-finiteness}). For \(H^{\ff}_{\w,q}\) one uses the Laplace type (b) on \(\mathcal S=X_K(K)\), which is the function \(Z_{H^{\ff}}(X_K,s)\) of \cref{subsec:height-zeta-entropy}; the Laplace type on the twists with \(H=H^{\prob}\) is the mass power sum \(Z^{\can}(s;\Fqr)\) of \cref{sec:7}, and with \(H=\log|\Aut(\xi)(\Fqr)|\) it is the sum \(\Theta_r(s)\) of \cref{sec:9}. The profile type (c) computes entropies.

For fixed \(r\) the Dirichlet-type series is finite and therefore entire. The dependence on \(r\) becomes relevant only when the extension degree varies. For comparison with height zeta functions over global fields, see \cite{ts-02,loeser-motzeta}.

\subsection{Stratified expression}
Using height strata \( \cH_m(X, \Fqr) := \{\x \in X(\Fqr) \mid H_{\w,q}(\x) = m\} \), the Dirichlet-type series over \( \Fqr \) can be rewritten as:
\[
Z_{H_{\w,q}}(X, s; q^r) = \sum_{m} \frac{|\cH_m(X, \Fqr)|}{m^s} = |X(\Fqr)| \sum_{m} \frac{\mu_{q^r}(m)}{m^s},
\]
where \( \mu_{q^r}(m) \) is the empirical height profile defined in \cref{sec:4}. Its logarithmic derivative at \(s=0\) is the mean logarithmic height,
\[
-\left.\frac{d}{ds}\right|_{s=0}\log Z_{H_{\w,q}}(X, s; q^r) = \sum_m\mu_{q^r}(m)\log m ,
\]
which is a moment of the profile and not its entropy. The entropy is obtained from the profile zeta function:
\[
\Ent_{q^r}(X) = - \sum_m \mu_{q^r}(m) \log \mu_{q^r}(m) = -\left.\frac{d}{ds}\right|_{s=1}\log\zeta_{\mu_{q^r}}(s),
\]
as in information theory \cite{renyi1961}, where \(\frac{1}{1-s}\log\zeta_\mu(s)\) is the R\'enyi entropy of order \(s\). For the probabilistic height on the twists, types (b) and (c) coincide, because \(e^{-H^{\prob}}=p_r\): the Laplace-type series of \(H^{\prob}\) is the profile zeta function of the mass measure. This coincidence is the content of \cref{thm:duality}.

\subsection{Asymptotic analysis}
For the degree-based height the limit distribution \( \mu(m) := \lim_{r \to \infty} \mu_{q^r}(m) \) is identically zero by \cref{prop:escape}, so an asymptotic series built from it vanishes. The appropriate object is of Laplace type and sums over all geometric points, which is possible by the Northcott property of \cref{Pp1}:
\[
Z_{H^{\deg}}^{\infty}(X, s) := \sum_{\x\in X(\overline{\F}_q)} q^{-s\,H^{\deg}_{\w,q}(\x)}.
\]

\begin{thm}%[Abscissa Bound for Degree-Based Height]
For the degree-based height \( H_{\w,q}^{\deg} \) on a hypersurface \(X\subset\bP^n_\w\), the series \( Z_{H^{\deg}}^{\infty}(X, s) \) converges absolutely for
\[
\Re(s) > (n-1)\min_i w_i .
\]
\end{thm}

\begin{proof}
The number of geometric points with \(r_\x=e\) is at most \(|X(\F_{q^e})|\le C q^{e(n-1)}\). By \cref{prop:deg-bounds} each such point has \(H^{\deg}_{\w,q}(\x)\ge e/\min_iw_i\). Hence the series is dominated by \(C\sum_{e\ge1} q^{e(n-1)-\Re(s)\,e/\min_i w_i}\), which converges for \(\Re(s)>(n-1)\min_iw_i\).
\end{proof}

For the function-field height the asymptotic analysis is that of \cref{subsec:height-zeta-entropy}. We record the Tauberian statement in the form adapted to function fields, where the zeta function is a power series and not a Dirichlet series with real exponents.

\begin{prop}
Suppose \( Z(u)=\sum_{k\ge0} N(k)u^k \), \(N(k)\ge0\), is meromorphic in a disc \(|u|<\rho^{-1}+\eta\) whose only pole there is a simple pole at \( u = \rho^{-1} \), \(\rho>1\). Then
\[
\sum_{k \leq T} N(k) \sim \frac{c\,\rho}{\rho-1}\, \rho^{T} \quad \text{as } T \to \infty \text{ through integers},
\]
for some constant \( c > 0 \). In particular, \( \rho \) governs the growth rate of points of bounded height on \( X \).
\end{prop}
\begin{proof}
As in \cref{prop:entropy-pole}, \(N(k)=c\rho^k+O(\rho_1^k)\) for some \(\rho_1<\rho\), and one sums the geometric series. For Dirichlet series with a single pole on the line of convergence the corresponding statement is the Wiener--Ikehara theorem (cf. \cite{korevaar-tauberian}); it does not apply directly here, since \(u=q^{-s/m}\) is periodic in \(s\).
\end{proof}

\begin{exa}
Let \( X \) be the hypersurface of \cref{exa:running}. Over \(\Fq\) the Dirichlet-type series of the degree-based height is \(Z_{H^{\deg}}(X,s;q)=q+1\), and over \(\F_{q^2}\) it is \((q+1)+(q^2-q)2^{-s}\). The Laplace-type series of the probabilistic height on the twists is \(Z^{\can}(s;\Fq)=(q+1)^{-s}(q+2^{1-s})\) for \(q\) odd.
\end{exa}

\begin{rem}
The probabilistic height \(H^{\prob}_{\w,q}\), due to its logarithmic definition, takes values in the interval \([\log|X(\Fqr)|,\log|X(\Fqr)|+\log\max_iw_i]\). Its Dirichlet-type series is therefore a finite sum over the isotropy values. The transform of \(H^{\prob}_{\w,q}\) used in \cref{sec:9} is instead the Laplace-type series \(\Theta_r(s)\), equivalently the mass power sum \(Z^{\can}(s;\Fqr)\).
\end{rem}

\subsection{Entropy identities}
Under the hypothesis of \cref{prop:entropy-pole}, the dominant pole of the function-field height zeta function determines the limiting entropy \(\lim_B\Ent_B^{\ff}(X)\).  For the finite-field mass measure, the logarithmic derivative of \(Z^{\can}(s;\Fqr)\) at \(s=1\) gives the groupoid entropy by \cref{thm:duality}.

%**********

\section{Sector factorization and applications}\label{sec:9}

We now collect the isotropy-weighted point counts into a generating series and prove the sector factorization.  The entropy identities and functional equations used later are consequences of this product formula.

Throughout, \(\mathfrak{X}= \big[ (\hat X \setminus \{0\}) \big/ \G_m \big]\) denotes the associated stack and \(X\) its coarse space. For \(r\ge1\) the groupoid \(\mathfrak{X}(\Fqr)\) has \(g_\x(r)\) isomorphism classes over each \(\x\in X(\Fqr)\), each with automorphism group of order \(g_\x(r)\). By \cref{prop:prob-mass} the probabilistic height is, up to the constant \(\log|X(\Fqr)|\), the \textbf{isotropy height} \(\log|\Aut(\xi)(\Fqr)|\). We define the height sums
\[
\Theta_r(s)
:= \sum_{\xi\in\pi_0(\mathfrak{X}(\Fqr))} |\Aut(\xi)(\Fqr)|^{-s}
= \sum_{\x\in X(\Fqr)} g_\x(r)^{1-s},
\]
so that \(Z^{\can}(s;\Fqr)=|X(\Fqr)|^{-s}\,\Theta_r(s)\) is the mass power sum of \cref{sec:7}. The \textbf{stacky height zeta function} is the formal series
\[
Z^{\can}_H(\mathfrak{X},s;t)
:=
\exp\Bigl(\sum_{r\ge1}\Theta_r(s)\frac{t^r}{r}\Bigr).
\]
For arbitrary complex \(s\), all powers of zeta functions below are interpreted in the formal group \(1+t\,\mathbb C[[t]]\) via the formal logarithm; equivalently, analytically near \(t=0\) one uses the branch with value \(1\) at \(t=0\).  Rationality in \(t\) is asserted only at the integral specializations for which the exponents are integers.

For a positive integer \(e\) prime to \(p\) let \(o_e=\ordq{e}\) be the multiplicative order of \(q\) modulo \(e\), with the convention \(o_1:=1\) (so that the condition \(o_1\mid r\) is vacuous), and let
\[
J_{1-s}(e) := \sum_{e'\mid e}\mu(e/e')\,(e')^{1-s}
\]
be the Jordan totient function of exponent \(1-s\), so that \(J_1=\phi\) and \(J_0(e)=0\) for \(e>1\).

The sums \(\Theta_r(s)\) are governed not by the locally closed strata of \(X\), but by the closed \textbf{sectors} cut out by the isotropy.

\begin{defn}\label{def:sector}
Let \(E_\w\) be the set of integers \(e\ge1\), prime to \(p\), dividing at least one \(w_i\). For \(e\in E_\w\) put
\[
L_e := \{\x : x_i=0 \text{ for every } i \text{ with } e\nmid w_i\},
\qquad
X^{(e)} := X\cap L_e ,
\]
a closed subscheme of \(X\) defined over \(\Fq\), and let \(\delta_e=\dim X^{(e)}\). Thus \(X^{(1)}=X\), and \(L_e\) is the fixed locus of \(\mu_e\subset\G_m\) acting with the weights \(\w\).
\end{defn}

When \(p\nmid w_i\) for all \(i\), \(X^{(e)}\) is the coarse space of the twisted sector of \(\mathfrak X\) indexed by a primitive \(e\)-th root of unity (\cref{prop:inertia}); in general it is the rational isotropy locus of \cref{rem:char}. The sectors inherit quasi-smoothness.

\begin{lem}\label{lem:sector-qs}
Assume \(p\nmid w_i\) for all \(i\) and let \(e\in E_\w\) with \(X^{(e)}\neq\varnothing\). Then \(X^{(e)}\) is a weighted hypersurface, or a weighted projective space, in \(\bP_{\w_e}\), where \(\w_e\) is the subvector of \(\w\) indexed by \(\{i : e\mid w_i\}\); and if \(X\) is quasi-smooth, so is \(X^{(e)}\).
\end{lem}

\begin{proof}
The cone \(\hat X\) is stable under \(\G_m\), hence under \(\mu_e\), and \(\hat X\cap L_e\) is the scheme-theoretic fixed locus \(\hat X^{\mu_e}\); it is the cone over \(X^{(e)}\) inside \(L_e=\bA^{\{i\,:\,e\mid w_i\}}\), where the equation of \(X\) restricts either to \(0\) or to a weighted homogeneous polynomial. Since \(e\mid w_i\) for some \(i\) and \(p\nmid w_i\), one has \(p\nmid e\).  After base change to \(\overline{\F}_q\), the group scheme \(\mu_e\) becomes a constant cyclic group of order \(e\).  The fixed-point scheme of a finite group of order invertible on a smooth scheme is smooth \cite[Prop.~3.4]{Edixhoven1992}.  Hence \((\hat X^{\mu_e}\setminus\{0\})_{\overline{\F}_q}\) is smooth, and therefore \(X^{(e)}\) is quasi-smooth.
\end{proof}

\begin{lem}\label{lem:sector-divis}
Let \(e\in E_\w\), \(r\ge1\) and \(\x\in X(\Fqr)\). Then \(e\mid g_\x(r)\) if and only if \(\x\in X^{(e)}(\Fqr)\) and \(o_e\mid r\). Consequently
\[
\Theta_r(s)=\sum_{e\in E_\w,\ o_e\mid r} J_{1-s}(e)\,\bigl|X^{(e)}(\Fqr)\bigr| .
\]
\end{lem}

\begin{proof}
Write \(S=S(\x)\). One has \(e\mid g_\x(r)=\gcd(k_S,q^r-1)\) if and only if \(e\mid k_S\) and \(e\mid q^r-1\). The first holds if and only if \(e\mid w_i\) for every \(i\in S\), that is, \(x_i=0\) whenever \(e\nmid w_i\), that is, \(\x\in X^{(e)}\); the second holds if and only if \(o_e\mid r\). Since \(g_\x(r)\) is prime to \(p\), M\"obius inversion in the form \(g^{1-s}=\sum_{e\mid g}J_{1-s}(e)\) gives
\(
g_\x(r)^{1-s}=\sum_{e\in E_\w}J_{1-s}(e)\,\mathbf{1}[e\mid g_\x(r)],
\)
and summing over \(\x\in X(\Fqr)\) yields the displayed formula.
\end{proof}

\begin{rem}[The inertia tower]\label{rem:inertia-tower}
The values \(\Theta_r(1-k)\), \(k\ge0\), are the masses of a tower of stacks. Since \(\Aut(\xi)(\Fqr)\) is the cyclic group \(\mu_{g_\x(r)}(\Fqr)\) for a twist \(\xi\) lying over \(\x\), any tuple of its elements commutes, so the \(k\)-fold inertia stack \(\mathbb I^k(\mathfrak X)\) of \cite{Morava-HKR}---whose objects over \(\xi\) are \(k\)-tuples \((g_1,\dots,g_k)\) in \(\Aut(\xi)\), with automorphism group again \(\Aut(\xi)\)---has
\[
\#^{\stack}\bigl(\mathbb I^k(\mathfrak X)(\Fqr)\bigr)
= \sum_{\xi\in\pi_0(\mathfrak X(\Fqr))} |\Aut(\xi)(\Fqr)|^{k-1}
= \sum_{\x\in X(\Fqr)} g_\x(r)^{k}
= \Theta_r(1-k).
\]
Because the stabilizers are abelian, the groupoid of \(\Fqr\)-points of \(\mathbb I(\mathbb I^k(\mathfrak X))\) agrees, object for object, with that of \(\mathbb I^{k+1}(\mathfrak X)\), so the identity \(\Theta_r(-k)=\Theta_r(1-(k+1))\) is, in the present setting, the instance for \(\mathcal Y=\mathbb I^k(\mathfrak X)\) of the general theorem of \cite{han-park}, that the unweighted point count of any finite-type algebraic stack with quasi-separated finite-type diagonal equals the mass of its inertia stack. \cref{thm:main} packages these counts into a single two-variable formal family: for every \(k\ge0\), the specialization \(s=1-k\) is a rational function of \(t\) with Weil-type bounds inherited from the closed sectors.  No rationality assertion is made for general nonintegral \(s\).
\end{rem}

The main statement is the following.

\begin{thm}[Main Theorem] % [Stacky Height--Entropy--Zeta Correspondence]
\label{thm:main}
Let \(X \subset \bP_\w^n\) be a weighted hypersurface over \(\Fq\), and let \(\mathfrak{X}\) be its associated stack.
\begin{enumerate}[label={\textup{(\roman*)}}]
\item
\textup{(Sector factorization.)} The stacky height zeta function is the finite product
\[
Z^{\can}_H(\mathfrak{X},s;t)
=
\prod_{e\in E_\w}
Z\bigl(X^{(e)}\otimes_{\Fq}\F_{q^{o_e}},\,t^{\,o_e}\bigr)^{J_{1-s}(e)/o_e},
\]
over the sectors of \cref{def:sector}, with \(Z(\varnothing,t)=1\). The logarithmic derivative \(t\,\partial_t\log Z^{\can}_H\) is a rational function of \(t\) with at most simple poles, each satisfying
\[
t_0^{-o_e}=\alpha,
\qquad\text{equivalently}\qquad
t_0=\bigl(\e\,\alpha^{1/o_e}\bigr)^{-1},\quad \e^{o_e}=1,
\]
where \(\alpha\) is a Frobenius eigenvalue on \(H^*_c\) of the sector \(X^{(e)}\otimes\F_{q^{o_e}}\), of some integer weight \(j\) with \(0\le j\le2\dim X\). Since \(|\alpha|=q^{o_ej/2}\), this gives \(|t_0|=q^{-j/2}\); but \(t_0^{-1}\) is an \(o_e\)-th root of such an eigenvalue times a root of unity, and need not itself be a Frobenius eigenvalue of \(X\).
For \(s\in\{1,0,-1,-2,\dots\}\) the exponents \(J_{1-s}(e)/o_e\) are integers and \(Z^{\can}_H(\mathfrak{X},s;t)\) is rational in \(t\), with
\[
Z^{\can}_H(\mathfrak{X},1;t)=Z^{\stack}(\mathfrak{X},t)=Z(X,t),
\qquad
Z^{\can}_H(\mathfrak{X},0;t)=Z^{\mathrm{tw}}(\mathfrak{X},t)=\prod_{e\in E_\w}Z\bigl(X^{(e)}\otimes\F_{q^{o_e}},t^{o_e}\bigr)^{\phi(e)/o_e},
\]
the zeta function of the coarse space and the twist zeta function of \cite{paper-1}.

\item
Let \(\w' = \w/d\) be the reduced weight vector, and let \(\mathfrak{X}'\) be the stack of the same hypersurface in \(\cP_{\w'}\). Then
\(
Z^{\can}_H(\mathfrak{X},1;t)
=
Z^{\can}_{H}(\mathfrak{X}',1;t),
\)
and the pushforward of the mass measure to the coarse space, the stacky height \(H^{\stack}\) and the stacky entropy \(\Ent^{\stack}_{q^r}\)
are invariant under the common reduction \(\w\mapsto\w/d\). For \(s\neq1\) one has \(\Theta_r(s)=\Theta'_r(s)\) whenever \(\gcd(d,q^r-1)=1\), and in general the two differ by the contribution of the \(\mu_d\)-gerbe \(\mathfrak X\to\mathfrak X'\).

\item
Let \(r\) be such that \(X(\Fqr)\ne\varnothing\). The value \(\Theta_r(1)\) equals the total stacky mass \(|X(\Fqr)|\) and can be divided out to produce
\(
\Theta^{\reg}_r(s)
=
\Theta_r(s)/\Theta_r(1).
\)
The groupoid entropy satisfies the identity
\[
\Ent^{\can}_{q^r}(\mathfrak{X})
=\log|X(\Fqr)|
- \left.\frac{d}{ds}\right|_{s=1}
  \log \Theta^{\reg}_r(s),
\]
and it is the mean of the probabilistic height with respect to the mass measure.
Thus the distribution of the \(\Fqr\)-twists by height is encoded in the behavior of the stacky height zeta function at \(s=1\).

\item
Assume in addition that \(X\) is geometrically irreducible of dimension \(n-1\ge1\). With \(\Sigma_r\), \(\varepsilon_r\), \(c_r\), \(d_X\), \(\kappa_r\), \(N_r\), and the period \(M\) of \cref{subsec:lang-weil-entropy}, one has \(\varepsilon_r=O(q^{-r})\) and \(\Ent^{\prob}_{q^r}(X)=O(rq^{-r})\).  For \(r\) with \(X(\Fqr)\ne\varnothing\) and \(\varepsilon_r\le1/2\),
\[
h(\varepsilon_r)\le \Ent^{\prob}_{q^r}(X)
\le h(\varepsilon_r)+\varepsilon_r\log(\kappa_r-1)
\]
when \(\kappa_r\ge2\), while the entropy is zero when \(\kappa_r=1\).  Moreover
\[
\Ent^{\can}_{q^r}(\mathfrak X)
=\log|X(\Fqr)|+\log\gcd(d_X,q^r-1)+O(q^{-rc_r}),
\]
and, whenever \(\Sigma_r\ne\varnothing\) and \(N_r\ge1\), along each residue class modulo \(M\),
\[
\varepsilon_r=N_rq^{-rc_r}\bigl(1+O(q^{-r/2})\bigr),\qquad
\Ent^{\prob}_{q^r}(X)=c_rN_r\,r\log q\,q^{-rc_r}\bigl(1+o(1)\bigr).
\]
The locus \(\Sigma_r\) is determined by the non-generic rational isotropy loci \(X^\e\), \(\e\in\Fqr^\times\), \(\e^{d_X}\ne1\); when \(p\nmid w_i\), these are the rational non-generic twisted sectors of the inertia stack of \(\mathfrak X\).
\end{enumerate}
\end{thm}

\noindent
Parts (i)--(iii) hold for every weighted hypersurface.  Part (iv) requires geometric irreducibility and positive dimension.  No restriction on the characteristic is needed for these four statements; see \cref{rem:char}.

\begin{proof}[Proof of \cref{thm:main}\textup{(i)}: sector factorization and Weil-type bounds]
By \cref{lem:sector-divis},
\[
\sum_{r\ge1}\Theta_r(s)\frac{t^r}{r}
=\sum_{e\in E_\w}J_{1-s}(e)\sum_{\substack{r\ge1\\ o_e\mid r}}\bigl|X^{(e)}(\Fqr)\bigr|\frac{t^r}{r}
=\sum_{e\in E_\w}\frac{J_{1-s}(e)}{o_e}\sum_{j\ge1}\bigl|X^{(e)}(\F_{(q^{o_e})^{j}})\bigr|\frac{(t^{o_e})^{j}}{j},
\]
where the second equality is the substitution \(r=o_ej\). The inner sum is \(\log Z(X^{(e)}\otimes\F_{q^{o_e}},t^{o_e})\), which gives the product. Each \(X^{(e)}\) is a projective variety over \(\Fq\), so its zeta function is rational by \cite{dwork1960}, and by the Grothendieck--Lefschetz trace formula its reciprocal zeros and poles are Frobenius eigenvalues on \(H^i_c(X^{(e)}_{\overline{\F}_q},\Ql)\), each pure of some integer weight \(j\le i\le 2\dim X^{(e)}\le 2\dim X\) by \cite{deligne-weil2}. Substituting \(t^{o_e}\) for the variable replaces a reciprocal zero \(\alpha\) of \(Z(X^{(e)}\otimes\F_{q^{o_e}},\cdot)\) by the \(o_e\) numbers \(\e\alpha^{1/o_e}\), \(\e^{o_e}=1\), of the same absolute value, whence the statement on the poles of the logarithmic derivative.

For the integrality, \(J_k(e)\) with \(k\ge1\) is the number of elements of order exactly \(e\) in \((\Z/e)^k\); multiplication by \(q\) acts freely on that set with orbits of length \(o_e\), so \(o_e\mid J_k(e)\). For \(s=1\) only \(e=1\) contributes, since \(J_0(e)=0\) for \(e>1\), and the product is \(Z(X,t)\), in agreement with \cref{thm:stack-zeta}. For \(s=0\) one has \(J_1=\phi\) and \(\Theta_r(0)=\#\pi_0(\mathfrak X(\Fqr))\), which is the definition of the twist zeta function \cite[Sec.~8.1]{paper-1}.
\end{proof}

\begin{proof}[Proof of \cref{thm:main}\textup{(ii)}: reduction of the weights]
At \(s=1\) the function is \(Z(X,t)\) by part (i), and the coarse varieties for \(\w\) and \(\w'\) coincide. The statements on the mass measure, the stacky height and the stacky entropy are \cref{prop:mass}, \cref{prop:stack-height-inv} and \cref{prop:stack-entropy-inv}. If \(\gcd(d,q^r-1)=1\) then \(\gcd(k_S,q^r-1)=\gcd(k_S/d,q^r-1)\) for all \(S\), so \(\Theta_r=\Theta_r'\). In general the quotient \(g_S(r)/g'_S(r)\) is the local gerbe index of \cref{COR3}, computed cohomologically in \cite[Cor.~4.3]{paper-1} when \(p\nmid d\); for the hypersurface of \cref{exa:running} and \(\w=(2,2,4)\), \(q\) odd, one has \(\Theta'_1(0)=q+2\) and \(\Theta_1(0)=2q+\gcd(4,q-1)\).
\end{proof}

\begin{proof}[Proof of \cref{thm:main}\textup{(iii)}: entropy--zeta identity]
The sums \(\Theta_r(s)\) are recovered from the zeta function by
\(
\Theta_r(s)=\frac{1}{(r-1)!}\,\partial_t^{\,r}\log Z^{\can}_H(\mathfrak X,s;t)\big|_{t=0}.
\)
From the definition, \(\Theta_r(1)=\sum_\x1=|X(\Fqr)|\), which is the total mass by \cref{prop:mass}, and
\[
-\left.\frac{d}{ds}\right|_{s=1}\log\Theta_r(s) = \frac{1}{|X(\Fqr)|}\sum_{\x\in X(\Fqr)}\log g_\x(r).
\]
By \cref{prop:gerbe-entropy} the right-hand side is \(\Ent^{\can}_{q^r}(\mathfrak X)-\log|X(\Fqr)|\). Equivalently, \(Z^{\can}(s;\Fqr)=|X(\Fqr)|^{-s}\Theta_r(s)\) and the identity is \cref{thm:duality}. The regularized sums \(\Theta_r^{\reg}(s)\) form the moment-generating function, in the variable \(1-s\), of the isotropy height \(\log g_\x(r)\) for the uniform distribution on \(X(\Fqr)\).
\end{proof}

\begin{proof}[Proof of \cref{thm:main}\textup{(iv)}: entropy asymptotics]
The bounds for \(\Ent^{\prob}_{q^r}(X)\) are \cref{thm:prob-asymptotics}. For the groupoid entropy, by \cref{prop:gerbe-entropy} it suffices to estimate the mean of \(\log g_\x(r)\). The points outside \(\Sigma_r\) have \(g_\x(r)=\gcd(d_X,q^r-1)\), and the remaining points, a proportion \(\varepsilon_r\) of \(X(\Fqr)\), have \(1\le g_\x(r)\le\max_iw_i\). Hence the mean differs from \(\log\gcd(d_X,q^r-1)\) by at most \(\varepsilon_r\log\max_iw_i\). The description of \(\Sigma_r\) by the loci \(X^\e\) is \cref{prop:inertia}.
\end{proof}

\subsection{Functional equations}\label{subsec:fe}

The product formula reduces the functional-equation problem to the corresponding equations for the sectors.  Equidimensionality is sufficient under self-duality, and it is also necessary under the additional purity and geometric-irreducibility hypotheses stated below.

Say that a projective variety \(Y\) of dimension \(\delta\) over a finite field \(\F_Q\) is \textbf{self-dual} if \(Z(Y,T)\) satisfies
\(
Z(Y,1/(Q^{\delta}T))=\pm\,Q^{\delta\chi/2}T^{\chi}Z(Y,T)
\)
with \(\chi\) the \(\ell\)-adic Euler characteristic.  Smooth projective geometrically irreducible varieties are self-dual by Poincar\'e duality.  The finite-quotient situation needed for the diagonal examples is proved directly in \cref{prop:diagonal-selfdual}.

\begin{thm}[Functional equation from equidimensional sectors]\label{thm:fe-criterion}
Let \(X\subset\bP^n_\w\) be a weighted hypersurface over \(\Fq\) of dimension \(\delta\), with \(p\nmid w_i\) for all \(i\), and suppose every nonempty sector \(X^{(e)}\otimes\F_{q^{o_e}}\), \(e\in E_\w\), is self-dual.  If every nonempty sector has dimension \(\delta\), then for each
\[
s\in\{1,0,-1,-2,\ldots\}
\]
one has
\[
Z_H^{\can}\!\left(\mathfrak X,s;\frac{1}{q^{\delta}t}\right)
=\pm q^{\delta N_s/2}t^{N_s}Z_H^{\can}(\mathfrak X,s;t),
\qquad
N_s:=\sum_{e\in E_\w}J_{1-s}(e)\,\chi\!\left(X^{(e)}\right).
\]
In particular, at \(s=0\),
\[
Z^{\mathrm{tw}}\!\left(\mathfrak X,\frac{1}{q^{\delta}t}\right)
=\pm q^{\delta N_0/2}t^{N_0}Z^{\mathrm{tw}}(\mathfrak X,t),
\qquad
N_0=\sum_{e\in E_\w}\phi(e)\,\chi\!\left(X^{(e)}\right).
\]
\end{thm}

\begin{proof}
Put \(Q_e=q^{o_e}\) and \(T=t^{o_e}\).  Under \(t\mapsto1/(q^{\delta}t)\), one has \(T\mapsto1/(Q_e^{\delta}T)\).  Hence the self-duality of \(X^{(e)}\otimes\F_{Q_e}\), together with \(\dim X^{(e)}=\delta\), transforms its zeta factor by
\[
Z\!\left(X^{(e)}\!\otimes\F_{Q_e},\frac{1}{Q_e^{\delta}T}\right)
=\pm Q_e^{\delta\chi_e/2}T^{\chi_e}
 Z\!\left(X^{(e)}\!\otimes\F_{Q_e},T\right).
\]
For the stated values of \(s\), the exponent \(J_{1-s}(e)/o_e\) is an integer.  Raising the preceding identity to this exponent and multiplying over \(e\) gives
\[
\sum_e\frac{J_{1-s}(e)}{o_e}\,o_e\chi_e
=\sum_eJ_{1-s}(e)\chi_e=N_s,
\]
which proves the formula.  The specialization \(s=0\) uses \(J_1=\phi\).
\end{proof}

\begin{cor}[Exact criterion under purity and geometric irreducibility]\label{cor:fe-exact}
Assume the hypotheses of \cref{thm:fe-criterion}, and suppose in addition that every nonempty sector is geometrically irreducible and that \(H^i_c\) of every sector is pure of weight \(i\).  Then for every fixed \(s\in\{0,-1,-2,\ldots\}\), the following are equivalent:
\begin{enumerate}[label=\textup{(\roman*)}]
\item every nonempty sector has dimension \(\delta\);
\item \(Z_H^{\can}(\mathfrak X,s;t)\) satisfies a functional equation centered at \(\delta\), of the form
\[
Z_H^{\can}\!\left(\mathfrak X,s;\frac{1}{q^{\delta}t}\right)
=C_s t^{N_s}Z_H^{\can}(\mathfrak X,s;t)
\]
for some nonzero constant \(C_s\).
\end{enumerate}
In particular, this gives a necessary and sufficient criterion for the twist zeta function at \(s=0\).
\end{cor}

\begin{proof}
The implication \textup{(i)}\(\Rightarrow\)\textup{(ii)} is \cref{thm:fe-criterion}.  Conversely, fix \(s\le0\).  Then \(J_{1-s}(e)>0\) for every \(e\ge1\).  For a geometrically irreducible sector \(Y=X^{(e)}\otimes\F_{q^{o_e}}\) of dimension \(\delta_e\), the factors coming from \(H^0_c(Y)\) give, after the substitution \(T=t^{o_e}\), a total pole multiplicity \(J_{1-s}(e)\) on the circle of reciprocal radius \(1\).  Purity prevents cancellation of these poles by odd-degree cohomology.  Likewise the top cohomology \(H^{2\delta_e}_c(Y)\) contributes total pole multiplicity \(J_{1-s}(e)\) on the circle of reciprocal radius \(q^{\delta_e}\), again without cancellation by odd-degree cohomology.

A functional equation centered at \(\delta\) carries the reciprocal-pole circle of radius \(1\) to the circle of radius \(q^{\delta}\), preserving total multiplicity.  The former receives the positive contribution \(\sum_eJ_{1-s}(e)\) from all nonempty sectors, whereas the latter receives \(\sum_{\delta_e=\delta}J_{1-s}(e)\).  Equality of these multiplicities is therefore possible only if no nonempty sector has \(\delta_e<\delta\).  Since \(X^{(1)}=X\) has dimension \(\delta\), this is exactly \textup{(i)}.
\end{proof}

The hypothesis of self-duality is automatic in the two families below only after the stated geometric hypotheses are imposed.  For \(X=\bP^n_\w\), every sector is a weighted projective space \(\bP_{\w_e}\), whose zeta function is that of an ordinary projective space of the same dimension \cite[Cor.~2.2]{paper-1}.  For weighted diagonal hypersurfaces, purity and self-duality follow from a tame Fermat cover as follows.

\begin{prop}[Purity and self-duality for diagonal sectors]\label{prop:diagonal-selfdual}
Let
\[
X=\Bigl\{\textstyle\sum_{i=0}^n c_ix_i^{D/w_i}=0\Bigr\}\subset\bP^n_\w,
\qquad c_i\in\Fq^\times,\quad w_i\mid D,\quad p\nmid D.
\]
Then every nonempty sector \(X^{(e)}\) is quasi-smooth, its \(\ell\)-adic cohomology is pure, and it is self-dual in the sense of \cref{thm:fe-criterion}.  Consequently \cref{thm:fe-criterion} applies to \(X\) whenever its nonempty sectors are equidimensional.  If, in addition, all nonempty sectors are geometrically irreducible, then \cref{cor:fe-exact} gives the converse as well.
\end{prop}

\begin{proof}
Each sector is obtained by setting to zero the coordinates outside \(\{i:e\mid w_i\}\), and the restricted equation is again diagonal.  Thus it is enough to treat \(X\).  Since \(p\nmid D\) and \(w_i\mid D\), every nonzero exponent \(D/w_i\) is prime to \(p\).  The partial derivatives have no common zero on the punctured affine cone, so \(X\) is quasi-smooth.

Consider
\[
\pi\colon\bP^n\longrightarrow\bP^n_\w,
\qquad [y_0:\dots:y_n]\longmapsto[y_0^{w_0}:\dots:y_n^{w_n}].
\]
It is finite and surjective and carries the smooth Fermat hypersurface
\[
F=\Bigl\{\textstyle\sum_i c_i y_i^D=0\Bigr\}\subset\bP^n
\]
onto \(X\).  After base change to \(\overline{\F}_q\), the fibers are the orbits of the finite abelian group \(G=(\prod_i\mu_{w_i})/\mu_d\), where \(d=\gcd(w_0,\ldots,w_n)\) and \(\mu_d\) is embedded diagonally; its order is prime to \(p\), and \(X_{\overline{\F}_q}=F_{\overline{\F}_q}/G\).  If \(\pi:F\to X\) denotes the quotient map, then \(\pi\) is finite, \((\pi_*\Ql)^G=\Ql\), and taking \(G\)-invariants is exact.  Therefore
\[
H^*_c(X_{\overline{\F}_q},\Ql)
\simeq H^*_c(F_{\overline{\F}_q},\Ql)^G.
\]
The cohomology of \(F\) is pure and satisfies Poincar\'e duality.  The pairing restricted to the invariant subspaces is nondegenerate: if an invariant class pairs nontrivially with some class, averaging the latter over \(G\) preserves that pairing.  Thus the invariant cohomology is pure and self-dual, and so is \(X\).  The same argument applies verbatim to every nonempty sector.  For the arithmetic of diagonal hypersurfaces in the unweighted case see \cite{gouvea-yui}.
\end{proof}

\begin{exa}[The ambient space]
For \(X=\bP^n_\w\) the sector \(X^{(e)}\) is \(\bP_{\w_e}\), of dimension \(\#\{i:e\mid w_i\}-1\), which equals \(n\) exactly when \(e\) divides every \(w_i\). So the sectors are equidimensional if and only if each \(e\in E_\w\) divides all of the \(w_i\), that is, if and only if the prime-to-\(p\) parts of \(w_0,\dots,w_n\) are all equal. This is precisely the criterion of \cite[Thm.~7.6]{paper-1}, recovered here from the geometry of the sectors.
\end{exa}

\begin{exa}[A quasi-smooth curve with an exceptional sector]\label{exa:fe-fail}
Let \(p\neq 2,3\) and \(X=\{x_0x_2+x_1^3+x_0^3=0\}\subset\bP(1,1,2)\), of weighted degree \(3\). Its partial derivatives are \(x_2+3x_0^2\), \(3x_1^2\), \(x_0\), which vanish simultaneously only at the origin, so \(X\) is quasi-smooth; it is a rational curve, and it passes through the \(\mu_2\)-point \([0:0:1]\). Here \(E_\w=\{1,2\}\), \(X^{(1)}=X\) has \(q^r+1\) points, and \(X^{(2)}=X\cap\{x_0=x_1=0\}=\{[0:0:1]\}\) is a single point, of dimension \(0\neq 1\). With \(o_2=1\),
\[
Z^{\mathrm{tw}}(\mathfrak X,t)=Z(X,t)\,Z(X^{(2)},t)=\frac{1}{(1-t)^2(1-qt)},
\]
which admits no functional equation under \(t\mapsto1/(qt)\), while \(Z(X,t)=1/((1-t)(1-qt))\) does. The number of twists is therefore \(q^r+2\) for every \(r\), compared with \(q^r+1\) coarse points.
\end{exa}

\begin{exa}[A diagonal surface]\label{exa:fe-diag}
Let \(p\) be odd and \(X=\{x_0^4+x_1^4+x_2^2+x_3^2=0\}\subset\bP(1,1,2,2)\), a weighted diagonal surface of degree \(4\); it is quasi-smooth since the exponents \(4,4,2,2\) are prime to \(p\). Then \(E_\w=\{1,2\}\) and \(X^{(2)}=\{x_0=x_1=0,\ x_2^2+x_3^2=0\}\subset\bP(2,2)\) consists of two points, rational over \(\Fqr\) precisely when \(q^r\equiv1\pmod 4\), so \(o_2=1\) and \(\delta_2=0\neq2\). Accordingly
\[
\#\pi_0(\mathfrak X(\Fqr))=|X(\Fqr)|+\bigl|X^{(2)}(\Fqr)\bigr|,
\]
and the second term equals \(2\) when \(-1\) is a square in \(\Fqr\), and \(0\) otherwise.  The sector dimensions are \(2\) and \(0\), so the product has no functional equation centered at dimension \(2\), although each sector is self-dual.
\end{exa}

\begin{rem}\label{rem:fe-necessity}
The purity hypothesis in \cref{cor:fe-exact} is used only to exclude cancellation of the extreme poles by odd-degree cohomology.  Geometric irreducibility ensures that each nonempty sector contributes exactly one \(H^0\) class and one top class before the multiplicity coming from the Jordan-totient exponent.  Without these hypotheses a more refined criterion would have to keep track of Frobenius orbits of components and possible cancellations; we do not claim such a criterion here.
\end{rem}

\begin{rem}
The constants implicit in (iv) are those of the Lang--Weil estimate for \(X\) and for the components of \(\Sigma_r\), and are controlled by the \(\ell\)-adic Betti numbers of these varieties through the trace formula. When \(p\nmid w_i\), these are the Betti numbers of the stack \(\mathfrak X\) and of its twisted sectors \cite{Behrend}.
\end{rem}

%*********************************************
%***************************************************
\section{Concluding remarks}\label{sec:10}

The main arithmetic object of this paper is the family \(Z_H^{\can}(\mathfrak X,s;t)\).  Its sector factorization reduces the isotropy-weighted coefficients \(\Theta_r(s)\) to ordinary point counts on finitely many closed subvarieties \(X^{(e)}\).  At the integral specializations \(s=1,0,-1,-2,\ldots\) this produces rational functions of \(t\): the coarse-space zeta function at \(s=1\), the twist zeta function at \(s=0\), and the generating functions for iterated-inertia masses at \(s=1-k\).  For general complex \(s\), the product is understood formally, or analytically near \(t=0\); no global rationality in \(t\) is asserted.

The factorization also separates the functional-equation problem into geometric sectors.  Self-duality together with equidimensionality is sufficient for a functional equation centered at \(\dim X\).  If the sectors are in addition pure and geometrically irreducible, \cref{cor:fe-exact} shows that equidimensionality is necessary for every integral specialization \(s\le0\), and hence gives an exact criterion for the twist zeta function within that class.  For weighted diagonal hypersurfaces with \(p\nmid D\), \cref{prop:diagonal-selfdual} verifies purity and self-duality by a finite Fermat cover; equidimensionality remains a genuine additional condition, as \cref{exa:fe-diag} shows.

The entropy statements can also be read from the same isotropy data.  The probabilistic height is the information content of a twist for the normalized mass measure, and the derivative of \(\Theta_r(s)\) at \(s=1\) recovers the groupoid entropy.  The non-generic isotropy entropy has a Lang--Weil leading term on arithmetic subsequences.  Its period is governed jointly by the isotropy periods and the Frobenius orbit lengths of the relevant geometric components.

Finally, the three height constructions play different roles.  The function-field height is the stable height associated with \(\Oo(1)\) on the ambient weighted projective stack in the sense of Ellenberg--Satriano--Zureick-Brown, and the corresponding unstable height has a Northcott property on stack objects by the bounded-conductor Kummer argument of \cref{lem:kummer-finite,thm:can-properties}.  The degree-based height is instead a finite-field coordinate-complexity statistic; its stack-normalized version is invariant under the common reduction of the weight vector, but we do not identify it with an intrinsic Weil height.  These three constructions will therefore remain notationally distinct.

Two questions remain outside the scope of the paper.  For nonintegral \(s\), the factorization defines \(Z_H^{\can}(\mathfrak X,s;t)\) formally, or analytically near \(t=0\), but we do not address continuation in \(t\).  For the function-field height, a rationality theorem for \(Z_{H^{\ff}}(X_K,s)\) would turn the conditional statement \cref{prop:entropy-pole} into an unconditional asymptotic.  We make no claim on either point.

%*********************************************
\bibliographystyle{amsplain}
\bibliography{sh-112}

\end{document}